\documentclass[11pt]{article}

\usepackage{amsthm, amssymb, bm, amsmath,mathrsfs}

\usepackage{soul, xcolor}
\usepackage{tikz-cd}
\usepackage{schemata}

\usepackage{bbm,graphicx}

\usepackage{accents}

\usepackage[pagebackref,colorlinks=true,urlcolor=blue,pdfpagemode=UseNone,linkcolor=blue,citecolor=blue]{hyperref}

\newtheorem{theorem}{Theorem}[section]
\newtheorem{corollary}[theorem]{Corollary}

\newtheorem{lemma}[theorem]{Lemma}
\newtheorem{proposition}[theorem]{Proposition}

\theoremstyle{definition}
\newtheorem{definition}[theorem]{Definition}
\newtheorem{remark}[theorem]{Remark}

\newtheorem{theoremlit}[theorem]{Theorem}

\newcommand \Cm { \mathbb{C}}
\newcommand \Dm { \mathbb{D}}
\newcommand \Hm { \mathbb{H}}
\newcommand \Rm { \mathbb{R}}
\newcommand \Nm { \mathbb{N}}
\newcommand \Sm { \mathbb{S}}
\newcommand \Zm { \mathbb{Z}}

\renewcommand{\d}{\mathrm{d}}

\let \Re \undefined
\DeclareMathOperator{\Re}{Re}

\let \Im \undefined
\DeclareMathOperator{\Im}{Im}

\renewcommand \L {{\cal L}}
\newcommand \D {{\cal D}}

\newcommand \G {{\cal G}}
\newcommand \Gh {{\cal G}_{\mathrm h}}
\newcommand \Ghbar {\overline{{\cal G}_{\mathrm h}}}

\renewcommand \H {{\cal H}}
\renewcommand \S {{\cal S}}
\newcommand \T {{\cal T}}

\newcommand \x { {\mathrm x}}
\newcommand \rmv { {\mathrm v}}

\newcommand \wtH {\widetilde{H}}

\newcommand \Psihf {\Psi_\mathrm{hf}}

\newcommand \bPsifh {\bar{\Psi}_{\mathrm{hf}}^{-1}}
\newcommand \bPsihf {\bar{\Psi}_{\mathrm{hf}}}

\newcommand \Cev {C_{\mathrm{ev}}^\infty}

\newcommand \pih {\pi_\mathrm{h}}
\newcommand \gammav {\gamma^{{\mathrm v}}}
\newcommand \tx {\tilde{x}}

\newcommand \Cal {C_{\alpha}^\infty}
\newcommand \Calp {C_{\alpha,+}^\infty}
\newcommand \Calm {C_{\alpha,-}^\infty}

\newcommand \add[1] {#1}

\newcommand \muh {\mu_{\mathrm h}}

\usepackage[shortlabels]{enumitem} % allows to customize the numbering style in enumerate
\usepackage{mathtools}
\mathtoolsset{showonlyrefs} % displays a number only for equations that have been referenced
\allowdisplaybreaks % allows aligned equation to be split between pages

\title{The hyperbolic X-ray transform: new range characterizations, mapping properties and functional relations}% 
\author{Nikolas Eptaminitakis\thanks{Institut f\"{u}r Differentialgeometrie, Leibniz Universit\"{a}t Hannover, Welfengarten 1, 30167 Hannover, Germany;\newline email: nikolaos.eptaminitakis@math.uni-hannover.de} \and Fran\c{c}ois Monard\thanks{Department of Mathematics, University of California, Santa Cruz CA 95064; email:fmonard@ucsc.edu} \and Yuzhou Zou\thanks{Department of Mathematics, Northwestern University, Evanston IL 60208; \newline email: yuzhou.zou@northwestern.edu}}
\date{}

\begin{document}
\maketitle

\begin{abstract}
We derive new singular value decompositions and range characterizations for the X-ray transform on the Poincar\'e disk, intertwining relations with distinguished differential operators of wedge type, and a surjectivity result for the backprojection operator. New functional settings are found which allow to sharply understand boundary behavior issues and invertibility settings. The approach mainly exploits analogous results obtained only recently in the Euclidean disk, together with the projective equivalence between the two models.
%This is the abstract of your article. It should not exceed \textbf{200} words and needs to be concise and factual. State the purpose of the research, the principal results, and conclusion.
\end{abstract}

%%%%%%%%%%%%%%%%%%%%%%%%%%%%%%%%%%%%%%%%%%%%%%%%%%%%%%
%                   6. BODY
%%%%%%%%%%%%%%%%%%%%%%%%%%%%%%%%%%%%%%%%%%%%%%%%%%%%%%

\section{Introduction} \label{sec:intro}

Consider the unit disk $\Dm_H = \{z\in \Cm, |z|\le 1\}$, with interior $\Dm_H^\circ$ equipped with the Poincar\'e metric
\begin{equation}\label{eq:Poincare_metric}
	g_H = c^{-2}(z)|\d z|^2 ,\qquad c(z) \coloneqq \frac{1-|z|^2}{2}, \index{$g_H$}\index{c}
\end{equation}
of constant curvature $-1$. The space of oriented geodesics through $(\Dm_H^\circ, g_H)$ can be modeled as the open manifold $\Gh = \add{(\Rm/2\pi\Zm)}_\beta\times \Rm_a$ arising in the ``horocyclic'' parameterization $\gamma_{\beta,a}(t)$ of unit-speed hyperbolic geodesics given in \eqref{eq:hyp_fan_beam}. Our main object of study is the {\em hyperbolic X-ray transform}
\begin{equation}
	I_0^H f (\beta,a) \coloneqq \int_\Rm f(\gamma_{\beta,a}(t))\ \d t, \qquad (\beta,a)\in \Gh,
    \label{eq:hypXrt}\index{$I_0$}
\end{equation}
initially defined for $f\in C_c^\infty(\Dm_H^\circ)$ and with values in $C_c^\infty (\Gh)$.

% prior literature
Such a transform is a classical object of study \cite{Helgason1990,Kurusa1991,Berenstein1991,Ishikawa1997,Lissianoi1997,Berenstein1999} that continues to receive recent attention, see \cite{Guillarmou2017,Richardson2025}; in particular, it is long known that the transform above is injective over integrands with compact support or in the Schwartz class, that its range over such integrands can be described in terms of moment conditions \cite{Berenstein1993a}, and a variety of global and semi-local inversion formulas can be derived \cite{Bal2005,Courdurier2013}. This transform is of interest due to its applications in imaging in media with variable refractive index, for its connections with the inverse conductivity problem \cite{Berenstein1994,Berenstein1996,Fridman1998}, and for geophysical data processing \cite{Chen2021}.

More recently, integral-geometric questions generalizing the study of \eqref{eq:hypXrt} have been considered on {\em asymptotically hyperbolic} geometries (geometric models generalizing $(\Dm_H^\circ, g_H)$ to non-constant curvature, see e.g. Section \ref{sec:cosphere}). This includes questions of injectivity of the associated X-ray transform and its connection with rigidity questions \cite{Graham2019,Eptaminitakis2021}, stability estimates \cite{Eptaminitakis2022}, and the study of associated non-abelian X-ray transforms \cite{Grebnev2023}.

% a number of unanswered questions remain
Formulating stability statements as in \cite{Eptaminitakis2022} begs the question of finding the appropriate Sobolev scales that sharply capture the mapping properties of an X-ray transform, or further yet, finding explicit functional settings, Fr\'echet or Sobolev, where a ``normal'' operator (i.e. the X-ray transform composed with its adjoint) can be made {\em invertible}. Such issues have only been looked at very recently, starting with the works \cite{Monard2017,Monard2019a}, as they lead, among other applications, to fulfilling crucial criteria in novel toolboxes in statistical inverse problems; see also \cite{Nickl2023}. In geometries without conjugate points, where the normal operator is a pseudo-differential operator in the interior, answering these questions requires to turn one's focus to the boundary geometry and identify the pseudo-differential calculus (factoring in boundary behavior) where the normal operator is invertible. Such constructions depend on the boundary model, whether incomplete convex, asymptotically hyperbolic, or asymptotically Euclidean/conic, etc. The first case is currently the most documented (see the recent topical review \cite{Monard2023b}), while the understanding of the other ones is still in its beginnings. This work represents a first step in this direction in the case of asymptotically hyperbolic geometries. Focusing on the Poincar\'e disk model as a first step is of relevance as it has the further desireable property that special normal operators (given below) can be found to be in the functional calculus of some distinguished {\em differential operators of wedge type} \add{(see Remark \ref{rmk:wedge})}, and this further informs which pseudo-differential algebra and Sobolev spaces to use.

% what we do here
One of the salient features of this article is that, on the topic of mapping properties of the hyperbolic X-ray transform, one can find invertibility settings and sharp functional settings to obtain tame and compact mapping properties, by leveraging (i) recent results established in the case of the Euclidean disk (an incomplete model), see \cite{Mishra2019,Mishra2022,Monard2019a}, and (ii) the projective equivalence between the Euclidean disk and the Poincar\'e disk (see Proposition \ref{prop:projEquiv} below). Such a projective equivalence is of course not new, and has been exploited in prior work, e.g. in \cite{Berenstein1993a}, to derive a range characterization via moment conditions, or in \cite{Palamodov2000a} to study partial data problems for a transform related to \eqref{eq:hypXrt}. Here, we exploit this equivalence further, notably paying close attention to the change in boundary behavior incurred through this projective equivalence which bridges a prototype of an {\em incomplete, convex} geometry with a prototype of a {\em complete, asymptotically hyperbolic} one. On the hyperbolic side, this requires compactifying $\Dm_H^\circ$ (in the obvious way) and $\Gh$ (in a manner dictated by the projective equivalence), which then allows to discuss function spaces with non-trivial behavior at infinity, where the hyperbolic X-ray transform enjoys sharp mapping properties, compactness, and good invertibility. A summary of the main results is as follows:

$\bullet$ In Theorem \ref{thm:SVDH}, we give several singular value decompositions (SVDs) for appropriately weighted versions of \eqref{eq:hypXrt}, in particular contexts where the associated normal operator is {\em compact}, a desirable feature for several applied purposes.

$\bullet$ Theorem \ref{thm:isomorphismH} provides a functional framework (Sobolev spaces on $\Dm_H$ and on the data space $\Gh$ that account for boundary behavior) where one can formulate sharp and global stability estimates for the transform \eqref{eq:hypXrt}.

$\bullet$ Corollary \ref{cor:adj_onto} gives a surjectivity result for the backprojection operator (defined in \eqref{eq:backproj} below); such a result is the hyperbolic analogue of \cite[Theorem 1.4]{Pestov2005}. The latter result, in the context of simple surfaces, is a crucial building block to several subsequent results in the analysis of X-ray and attenuated X-ray transforms, see \cite{Paternain2023}.

$\bullet$ Theorem \ref{thm:moments} gives a generalization of the range characterization of $I_0^H$ in terms of {\em hyperbolic moment conditions} (initially formulated in \cite{Berenstein1993a} for functions of Schwartz class), to functions with non-trivial boundary behavior and finite Sobolev regularity.

$\bullet$ Proposition \ref{prop:interH} and Theorem \ref{thm:FuncRelH} give new functional relations with distinguished differential operators on $\Dm_H$  and $\Gh$, where the differential operators $\L_{\gamma}^H$ on $\Dm_H$ (see \eqref{eq:L_wedge}) and $\T_\gamma^H$ on $\Gh$ (see \eqref{eq:TgH}) when $\Gh$ is suitably compactified, are of {\em wedge} type (see e.g.,  \cite{Schulze1991}). This is in contrast with prior work on the topic, where relations between \eqref{eq:hypXrt} and the hyperbolic Laplacian have been found (see e.g., \cite{Wallace1986,Guillarmou2017}), and where prior studies of the X-ray transform on asymptotic hyperbolic spaces make systematic use of the uniformly degenerate (or 0-) calculus of Mazzeo-Melrose (\cite{Mazzeo1987}), the natural calculus where the hyperbolic Laplacian belongs.

$\bullet$ For one specific weighted X-ray transform (the case $\gamma=0$ below), the functional relations just discussed simplify, and the infinite-dimensional co-kernel of the X-ray transform can be identified with the help of boundary operators $P_-^H$, $C_-^H$, analogous to those defined on simple surfaces in \cite{Pestov2004,Monard2015a,Mishra2019}. This leads to the refined range characterization given in Theorem \ref{thm:gammazeroHalt}, some of which required a refinement (to Sobolev scales) of the range characterization of the X-ray transform \cite[Thm 4.4.(ii)]{Pestov2004} in the case of the Euclidean disk (see Sec.\ \ref{sec:gammazero}).

The current results open the way for further studies, which the authors plan to address in future work, including the study of analogous questions for tensor fields (where the study of ``tensor fields with boundary behavior'', viewed as functions on the sphere bundle $S\Dm_H^\circ$, see \eqref{eq:SDH}, requires discussing an appropriate compactification of the latter), and an in-depth study of the link between the X-ray transform on asymptotically hyperbolic manifolds and wedge operators.

We now discuss the main results in the next section and give an outline of the remainder at the end of it.

%%%%%%%%%%%%%%%%%%%%%%%%%%%%%%%%%%%%%%%%%%%%%%%%%%%%%%%%%%%%%%%%%%%%%%%%%%%%%%%%%%%%%%%%%%%%%%%%%%%%%%%%%%%%
\section{Main results} \label{sec:main}

\subsection{Preliminaries} \label{sec:mainprelim}

On $\Dm_H$ equipped with its standard smooth structure, we define the distinguished boundary defining function (`bdf', here and below)
\begin{equation}
	x(z) \coloneqq \frac{1-|z|^2}{1+|z|^2}, \quad z\in \Dm_H.
	\label{eq:xbdf}\index{$x$}
\end{equation}
With $\omega = \arg(z)$\index{$\omega$} and $x$ as in \eqref{eq:xbdf}, the measure induced by the Riemannian volume form $\d V_H$ of $(\Dm_H^\circ, g_H)$ is given by $ x^{-2} \d x \d\omega$.

\subsubsection{Geodesics of $(\Dm^\circ_H,g_H)$}

There are several models for the manifold of oriented geodesics on $(\Dm_H^{\circ},g_H)$. Throughout the presentation, we will focus on a parameterization of the space of oriented geodesics by means of the manifold $\Gh = \add{(\Rm/2\pi\Zm)}_\beta\times \Rm_a$ by writing
\begin{equation}
	\gamma_{\beta,a}(t) \coloneqq  e^{i\beta} \frac{(2+ia) \x(t) + ia}{ia \x(t) - 2+ia}, \quad \x(t)\coloneqq \tanh \Big( \frac{t}{2} \Big), \quad (\beta,a,t)\in \add{(\Rm/2\pi\Zm)} \times \Rm\times \Rm.
	\label{eq:hyp_fan_beam}
\end{equation}
This parameterization arises from horocycles as we explain in Section\ \ref{sec:horocycle}. Orientation reversion arises via the involution
\begin{equation*}
    \S_A^H\colon \Gh\to \Gh, \qquad (\beta,a)\mapsto (\beta+\pi + 2\tan^{-1}a, -a).
    % \label{eq:SAH}
\end{equation*}
The map $\S_A^H$ arises as the composition of the scattering map with the antipodal map, see Section\ \ref{ssec:scattering_map} for details. Some results will also be given in vertex parameterization \eqref{eq:Hgeo} for completeness.

\subsubsection{Hyperbolic X-ray transform}

Recall the definition of $I_0^H$ in Eq. \eqref{eq:hypXrt}. If $f\in C_c^\infty(\Dm_H)$, then the smoothness of $I_0^H f$ comes from the smoothness of the map $(\beta,a,t)\mapsto \gamma_{\beta,a}(t)$. To see why $I_0^H f$ has compact support, we define on $\Gh$ the function
\begin{equation}
	\muh(\beta,a) = \muh(a)\coloneqq \cos(\tan^{-1}a) = (1+a^2)^{-1/2}.
	\label{eq:mu}\index{$\muh$}
\end{equation}
Then a direct calculation shows that with $\gamma_{\beta,a}(t)$ as in \eqref{eq:hyp_fan_beam}, we have
\begin{equation}
	x(\gamma_{\beta,a}(t)) \le \muh(a), \qquad t\in \Rm,
	\label{eq:xcontrol}
\end{equation}
see e.g. \eqref{eq:xhoro}. In particular, if $f = 0$ on $x^{-1} (0,\varepsilon)$, where $0<\varepsilon<1$, then $I_0^H f =0$ on $\muh^{-1}(0,\varepsilon)$. Because $I_0^H f$ does not depend on orientation, we obtain the symmetry condition
\begin{equation*}
    (I_0^H f) \circ \S_A^H = I_0^H f \qquad \text{on }\quad  \Gh.
    % \label{eq:sym}
\end{equation*}

\subsubsection{Hyperbolic backprojection operator}
The unit sphere bundle of $(\Dm_H^{\circ}, g_H)$ is denoted
\begin{equation}
	S\Dm_H^{\circ} = \{(z,w)\in T\Dm_H^{\circ}\colon g_H (w,w) = 1\}.
	\label{eq:SDH}
\end{equation}
In Section \ref{sec:horocycle}, we define a smooth map
\begin{equation}
	\pih \colon S\Dm_H^{\circ} \to \Gh,
	\label{eq:F}
\end{equation}
mapping $(z,w)\in S\Dm_H^{\circ}$ to the unique oriented geodesic in $\Gh$ passing through $(z,w)$. With $\pih$ defined in \eqref{eq:F}, for $g\in C^\infty(\Gh)$, we define the hyperbolic backprojection of $g$ by
\begin{equation}
	(I_0^H)^\sharp g (z) \coloneqq \int_{S_z\Dm_H^{\circ}} g(\pih(z,w))\ \d S_z(w),
	\label{eq:backproj}
\end{equation}
where $\d S_z(w)$ is the Lebesgue measure on the unit tangent circle $S_z\Dm_H^{\circ}$.

\subsection{Extension of \texorpdfstring{$I_0^H$}{I 0 H} to non-compactly supported integrands}\label{sec:extendedI0}

We first explain how to extend definition \eqref{eq:hypXrt} to spaces of non-compactly supported smooth functions. To this end, it will be useful to compactify $\Gh$ into $\Ghbar \coloneqq \add{(\Rm/2\pi\Zm)}_\beta \times [-\infty,\infty]_a$, where the factor $[-\infty,\infty]_a$ is given the smooth structure of a compact manifold with boundary using the stereographic projection map so that the function $\muh$ defined in \eqref{eq:mu}, extended by zero on $\partial\Ghbar$, becomes a smooth bdf.

The manifold $\Ghbar$ is then diffeomorphic to the inward-pointing boundary of the Euclidean disk $(\Dm_E = \{z\in \Cm, |z|\le 1\}, |dz|^2)$, denoted $\partial_+ S\Dm_E = \add{(\Rm/2\pi\Zm)}_\beta \times [-\pi/2,\pi/2]_\alpha$ (with bdf $\mu(\beta,\alpha) = \cos\alpha$), via the horocycle-to-fan-beam map
\begin{equation}
    \Psihf \colon \Ghbar \to \partial_+ S\Dm_E, \qquad (\beta,a)\mapsto (\beta,\tan^{-1} a),
	\label{eq:Psihf}
\end{equation}
and where, by convention, $\tan^{-1} (\pm \infty) = \pm \frac{\pi}{2}$.

On the Euclidean disk, an important subspace of $C^\infty(\partial_+ S\Dm_E)$ for capturing the range of the X-ray transform is $\Calp(\partial_+ S\Dm_E)$, defined in \eqref{eq:CalphappE}. Accordingly, on $\Ghbar$ it will be crucial to define
\begin{equation}
	\Calp(\Ghbar) \coloneqq \Psihf^* \Calp(\partial_+ S\Dm_E),
	\label{eq:CalphaG}
\end{equation}
\add{where $\Psihf^*$ denotes the pullback by $\Psihf$.} The following characterization holds.
\begin{proposition} \label{prop:charac}
    A function $h$ on $\Ghbar$ belongs to $\Calp(\Ghbar)$ if and only if $h\in C^\infty(\Ghbar)$, $h=h\circ \S^H_A$, and when expressed in terms of the coordinates $(\beta + \tan^{-1}a, \muh)$ with $\muh$ as in \eqref{eq:mu}, the odd terms of the Taylor expansion of $h$ off of $\partial\Ghbar$ all vanish.
\end{proposition}

Let $\Cev(\Dm_H)$\index{$\Cev(\Dm_H)$} be the space of smooth functions on $\Dm_H$ admitting Taylor expansions off of $\Dm_H$ only made of even terms when written in terms of the special bdf $x$ defined in \eqref{eq:xbdf}. Then, a first observation is the following.

\begin{theorem}\label{thm:smoothmapping}
    For any $\gamma>-1$, we have
    \begin{equation}
	I_0^H (x^{2+2\gamma} \Cev (\Dm_H)) \subset \muh^{2\gamma+2} \Calp (\Ghbar).
	\label{eq:smoothmapping}
    \end{equation}	
\end{theorem}

\add{
\begin{remark}
    The  range of $\gamma$'s above is the largest possible in which \eqref{eq:smoothmapping} can hold. Indeed, it is easy to see that \eqref{eq:smoothmapping} fails when $\gamma = -1$, as $I_0^H (1)$ is not defined due to the completeness of the hyperbolic geodesic flow. The exponent is written in the form $2\gamma+2$ to facilitate the transfer of results that hold on the Euclidean disk via projective equivalence.
\end{remark}
}

\begin{remark} The result for this even structure implies mapping properties on smooth spaces. Indeed, combining \eqref{eq:smoothmapping} with the decomposition
\begin{align*}
C^\infty(\Dm_H) = \Cev(\Dm_H) + x \Cev(\Dm_H),
\end{align*}
we obtain
	\begin{align*}
	    I_0^H (x^{2+2\gamma} C^\infty(\Dm_H)) & \subset \muh^{2\gamma+2} \Calp (\Ghbar) + \muh^{2\gamma+3} \Calp (\Ghbar) \\
              &= \muh^{2\gamma+2} \{u\in C^\infty(\overline{\Gh}): u= \add{u \circ \S_A^H}\}.
	\end{align*}
\end{remark}

\begin{remark}
    The description of the spaces $\Cev(\Dm_H)$ and $\Calp (\Ghbar)$ in terms of constrained Taylor expansions is not invariant under change of bdf. They can be described invariantly if one discusses an {\em even structure} on the underlying manifolds with boundary, see e.g. \cite{Eptaminitakis2021}. In the present case, we will not discuss the issue further, and confine ourselves to the distinguished bdf's $x$ and $\muh$ defined in \eqref{eq:xbdf} and \eqref{eq:mu} (the even structure determined by the distinguished bdf $x$ is the same as the one determined by the Poincaré metric, see \cite[Sec.\ 2]{Eptaminitakis2021}).
\end{remark}

\subsection{Singular value decompositions}

We now turn toward mapping properties of $I_0^H$ between Hilbert spaces. Below we write
\begin{equation}
    L^2_\pm (\Gh, \muh^{-2\gamma}\d\beta\d a) \coloneqq L^2(\Gh, \muh^{-2\gamma}\d\beta\d a) \cap \ker ( (\S_A^H)^* \mp id),
    \label{eq:Lpm}
\end{equation}
\add{where $(\S_A^H)^*$ denotes the pullback by $\S_A^H$.}

% bounded setting
\begin{theorem}\label{thm:bounded}
	Fix $\gamma>-1$. Then, the operator
	\begin{equation}
		I_0^H x^{2+2\gamma} \colon L^2(\Dm_H, x^{2\gamma+3}\ \d V_H) \to L_+^2(\Gh, \muh^{-2\gamma}\ \d\beta\d a),
		\label{eq:I0Hdomain}
	\end{equation}
	is bounded, with Hilbert space adjoint $(I_0^H x^{2+2\gamma})^* = x^{-1} (I_0^H)^\sharp\ \muh^{-2\gamma}$.
\end{theorem}

\begin{remark}
    Theorem \ref{thm:bounded} can be viewed as a generalization in this setting of \cite[Cor. 4.5]{Eptaminitakis2022}. Indeed, combining the statement there with the discussion in Section \ref{sec:cosphere}, one sees that $I_0^H:x^{\delta}L^2(\Dm_H,\d V_H)\to \muh ^{\delta'}L^2_+(\Gh,\d\beta\d a)$ is bounded provided $\delta'<\delta$, $\delta'<0$, and $\delta>-1/2$. Setting $\delta=\gamma+1/2$ and $\delta'= \gamma$, one obtains \eqref{eq:I0Hdomain} if $\gamma\in(-1,0)$.
\end{remark}

% SVD
We now give a full singular value decomposition of the operator \eqref{eq:I0Hdomain}.
As it will be inherited from that of the X-ray transform on the Euclidean disk, we recall some notation. Let us define the family of functions, for $n\in \Nm_0$ and $k\in \Zm$,
\begin{equation}
    \psi_{n,k}^\gamma (\beta,\alpha)= (\cos\alpha)^{2\gamma+1} e^{i(n-2k)(\beta+\alpha+\pi/2)} p_n^{\gamma} (\sin\alpha),
    \label{eq:psinkgamma}\index{$\psi_{n,k}^\gamma$} \quad \add{(\beta,\alpha)\in \partial_+ S\Dm_E},
\end{equation}
where $p_n^\gamma:[-1,1]\to \Rm$ is a scalar multiple of the $n$-th Jacobi polynomial\footnote{Commonly denoted $P^{(\gamma+1/2,\gamma+1/2)}_n$ in the literature. Up to a scalar, this is also the ultraspherical/Gegenbauer polynomial $C^{(\gamma+1)}_n$.}, orthogonal for the weight $(1-x^2)^{\gamma+1/2}$, normalized here for convenience so that $\int_{-1}^1 (p_n^\gamma(x))^2 (1-x^2)^{\gamma+1/2}\d x = (2\pi)^{-1}$. This makes $\{\psi_{n,k}^\gamma\}_{n\ge 0,\ k\in \Zm}$ an {\em orthonormal} basis for the space $L_+^2(\partial_+ S\Dm_E, \mu^{-2\gamma}\d\beta\d\alpha)$.

With $(I_0^E)^\sharp$, the backprojection operator associated with the Euclidean disk, we define
\begin{equation}
    Z_{n,k}^\gamma \coloneqq (I_0^E)^\sharp \mu^{-2\gamma-1} \psi_{n,k}^\gamma, \qquad n\ge 0, \qquad 0\le k\le n.
    \label{eq:Zernike}\index{$Z_{n,k}^\gamma$}
\end{equation}
Those are equal, up to normalization, to the generalized disk Zernike polynomials ($P_{n-k,k}^\gamma$ in the convention of \cite{Wuensche2005}), and they are a Hilbert basis for $L^2(\Dm_E,d^\gamma \d V_E)$, where
\begin{equation}
    d(z) \coloneqq 1-|z|^2 \in C^\infty(\Dm_E),
    \label{eq:d}
\end{equation}
and $\d V_E$ is the Euclidean volume form. An important map for what follows is the following:
\begin{equation}
    \Phi\colon \Dm_H\to \Dm_E, \qquad \Phi(z) \coloneqq \frac{2z}{1+|z|^2}.
    \label{eq:Phi}\index{$\Phi$}\index{$\Dm_E$}
\end{equation}
The map $\Phi$ is an isometry between $(\Dm_H^{\circ},g_H)$ and $(\Dm_E^{\circ},g_K)$, where $g_K$ is the Beltrami-Klein hyperbolic metric on the unit disk. The latter is projectively equivalent to the Euclidean metric, in the sense that the geodesics of the two metrics agree up to parameterization. It can be seen that the map $\Psihf$ in \eqref{eq:Psihf} is induced by $\Phi$, see Section \ref{ssec:models}, and that with $x$ and $d$ as in \eqref{eq:xbdf} and \eqref{eq:d}, one has
\begin{equation}\label{eq:bdfs}
    \Phi^* d = x^2.
\end{equation}
This implies that $\Phi^*:C^\infty(\Dm_E)\to \Cev(\Dm_H)$ is an isomorphism. Returning to $\Gh$, the family
\begin{equation}
    \psi_{n,k}^{\gamma,H} \coloneqq \muh \Psihf^*\ \psi_{n,k}^\gamma, \qquad n\ge 0, \quad k\in \Zm,
    \label{eq:psinkgammaH}
\end{equation}
is then a Hilbert basis of $L^2_+ (\Gh, \muh^{-2\gamma}\d\beta\d a)$ defined in \eqref{eq:Lpm}.

\begin{theorem}\label{thm:SVDH}
    Fix $\gamma>-1$.
    \begin{enumerate}[(a)]	
	\item The operator $(I_0^H x^{2+2\gamma})^* \colon L^2_+ (\Gh, \muh^{-2\gamma}\d\beta\d a) \to L^2(\Dm_H, x^{2\gamma+3}\ \d V_H)$ has kernel
	    \begin{equation*}
		\ker (I_0^H x^{2+2\gamma})^* = \mathrm{span} \left( \psi_{n,k}^{\gamma,H},\quad n\ge 0,\ k<0 \text{ or } k>n \right).
		% \label{eq:keradjointH}
	    \end{equation*}

	\item The singular value decomposition (SVD) of
	    \begin{equation}
		I_0^H x^{2+2\gamma} \colon L^2(\Dm_H, x^{2\gamma+3}\ \d V_H) \to L_+^2(\Gh, \muh^{-2\gamma}\d\beta\d a) / \ker (I_0^H x^{2+2\gamma})^*	
	    \end{equation}
	    is given by $\left(\widehat{\Phi^* Z_{n,k}^\gamma}, \psi_{n,k}^{\gamma,H}, \sigma_{n,k}^\gamma\right)_{n\ge 0,\ 0\le k\le n}$.
	    Here, $Z_{n,k}^\gamma$, $\Phi$, $\psi_{n,k}^{\gamma,H}$ are respectively defined in \eqref{eq:Zernike}, \eqref{eq:Phi}, and \eqref{eq:psinkgammaH}, and the positive real numbers $\sigma_{n,k}^{\gamma}$ satisfy
	    \begin{equation}
		(\sigma_{n,k}^{\gamma})^2 \coloneqq \frac{2^{2\gamma+2} \pi}{n+1} \frac{B(n-k+1+\gamma, k+1+\gamma)}{B(n-k+1, k+1)},
		\label{eq:signkgammaH}\index{$\sigma_{n,k}^{\gamma})^2$}
	    \end{equation}
	    with $B(x,y) = \frac{\Gamma(x) \Gamma(y)}{\Gamma(x+y)}$ denoting the Beta function. Here and later, the notation $\widehat{\cdot}$ stands for the normalized function with respect to the corresponding norm.
    \end{enumerate}
\end{theorem}

\add{
\begin{remark}
    The convention of studying $I_0^H x^{2+2\gamma}$ on $L^2(\Dm_H, x^{2\gamma+3} \d V_H)$ (instead of the equivalent $I_0^H$ on $L^2(\Dm_H, x^{-2\gamma-1} \d V_H)$) is chosen so that the normal operators $(I_0^H x^{2+2\gamma})^* I_0^H x^{2+2\gamma}$ are all isomorphisms of the reference space $\Cev(\Dm_H)$, see Theorem \ref{thm:isomorphismH} below.
\end{remark}
}

\add{
\begin{remark}
    Similar SVD results exist for the exponentially weighted X-ray transform on $\Rm^2$, see \cite{Davison1981}. In spite of the difference in geometry at infinity, these two examples have in common that they are formulated in weighted $L^2$ spaces where the weight decays exponentially as unit-speed geodesic time goes to $\pm \infty$.
\end{remark}
}

\subsection{Distinguished differential operators and special relations}

On $\Dm_H$, define the differential operator
\begin{equation}
	\begin{aligned}
    \L_\gamma^H &= - x^{-2\gamma-1} \partial_x \left( x^{2\gamma+1} (1-x^2) \partial_x \right) - \frac{1}{1-x^2} \partial_\omega^2 + (1+\gamma)^2 id, \nonumber \\
    &=-  (1-x^2) \partial_x^2 - \left( \frac{2\gamma+1}{x} - (2\gamma+3)x \right)\partial_x - \frac{1}{1-x^2} \partial_\omega^2 + (1+\gamma)^2 id.
\end{aligned}
    \label{eq:L_wedge}\index{$\L_\gamma^H$}
\end{equation}
\add{ \begin{remark}\label{rmk:wedge}
    The operator $x^2 \L_\gamma^H$ is easily seen to be a second-order {\em 0-differential} operator, i.e., a linear combination  of the family $\{(x\partial_x)^a (x\partial_\omega)^b\}_{a,b\ge 0,\, a+b\le 2}$ with smooth coefficients, originally introduced in the work \cite{Mazzeo1987} to study elliptic problems on (asymptotically) hyperbolic manifolds. By definition, this makes $\L_{\gamma}^H$ a second-order {\em wedge} operator, as studied in \cite{Schulze1991} and usually suited for elliptic problems arising in wedge geometry.
\end{remark}
}

One may find that the space $\Cev(\Dm_H)$ described in Section \ref{sec:extendedI0} is a subspace of $L^2(\Dm_H, x^{2\gamma+3}\d V_H)$ if and only if $\gamma>-1$, that $\L_\gamma^H (\Cev(\Dm_H))\subset \Cev(\Dm_H)$, and that $(\L_\gamma^H, \Cev(\Dm_H))$\footnote{\add{Here, $(\L_\gamma^H, \Cev(\Dm_H))$ denotes the operator $\L_\gamma^H$ with domain $\Cev(\Dm_H)$.}} is symmetric with respect to the inner product $L^2(\Dm_H, x^{2\gamma+3}\d V_H)$. More specifically, we have the following proposition.
\begin{proposition}\label{prop:LgH}
    Fix $\gamma>-1$. The operator $(\L_\gamma^H, \Cev(\Dm_H))$ acting on the space $L^2(\Dm_H, x^{2\gamma+3}\d V_H)$ is essentially self-adjoint (with closure denoted $\L_\gamma^H$ below), with spectral decomposition
    \begin{equation*}
	\Phi^* Z_{n,k}^\gamma, \quad (n+1+\gamma)^2, \qquad n\ge 0,\quad 0\le k\le n.
	% \label{eq:HZernike}
    \end{equation*}
\end{proposition}

On $\Gh$, define the differential operator $T \coloneqq \partial_\beta - (1+a^2) \partial_a$\index{$T$}, and
\begin{equation}
    \T_\gamma^H = -T^2 + 2(\gamma+1) a T + (\gamma^2 - 2(\gamma+1)a^2 - 1) id,
    \label{eq:TgH}\index{$\mathcal{T}_\gamma^H$}
\end{equation}
so that $\T_\gamma^H = \muh \circ \Psihf^{*} \circ \T_\gamma \circ \Psihf^{-*} \circ \muh^{-1}$, where $\T_\gamma$ is the operator on $\partial_+S \Dm_E$ defined in \eqref{eq:T_gamma_Eucl}.
Then, if $\gamma>-1$, one finds  that $\muh^{2\gamma+2} \Calp (\Ghbar)\subset L^2_+ (\Gh, \muh^{-2\gamma}\d\beta\d a)$, that $\T_\gamma^H (\muh^{2\gamma+2} \Calp(\Ghbar))\subset \muh^{2\gamma+2} \Calp(\Ghbar)$ (using \cite[eq. (28)]{Mishra2022}), and that the operator $(\T_\gamma^H, \muh^{2\gamma+2} C_{\alpha,+}^\infty(\Ghbar))$ is symmetric with respect to $L^2_+ (\Gh, \muh^{-2\gamma}\d\beta\d a)$. More specifically, we have the following proposition.
\begin{proposition}\label{prop:TgH}
    Fix $\gamma>-1$. The operator $(\T_\gamma^H, \muh^{2\gamma+2} \Calp(\Ghbar))$ acting on $L^2_+ (\Gh, \muh^{-2\gamma}\d\beta\d a)$ is essentially self-adjoint (with closure denoted $\T_\gamma^H$ below), with spectral decomposition
    \begin{equation*}
	\psi_{n,k}^{\gamma,H}, \quad (n+1+\gamma)^2, \qquad n\ge 0,\quad k\in \Zm.
	% \label{eq:Hpsinkgamma}
    \end{equation*}	
\end{proposition}

\begin{proposition}\label{prop:interH}
    We have the following intertwining properties:
    \begin{align}
	(I_0^H x^{2+2\gamma})^* \circ \T_\gamma^H &= \L_\gamma^H \circ (I_0^H x^{2+2\gamma})^* \quad \text{on} \quad \muh^{2\gamma+2} \Calp(\Ghbar),
	\label{eq:interadjH} \\
	I_0^H x^{2+2\gamma} \circ \L_\gamma^H &= \T_\gamma^H \circ I_0^H x^{2+2\gamma} \quad \text{on} \quad \Cev(\Dm_H).
	\label{eq:interH}
    \end{align}	
\end{proposition}
Combining \eqref{eq:interH} and \eqref{eq:interadjH}, we see that the operator $(I^H_0 x^{2+2\gamma})^* I^H_0 x^{2+2\gamma}$ commutes with $\L_\gamma^H$, and by symmetry arguments, also with the $L^2(\Dm_H, x^{2\gamma+3}\d V_H)$-graph closure of $(-\partial_\omega^2, \Cev(\Dm_H))$, denoted $-\partial_\omega^2$ below.

Given the spectrum of each of the operators $(I^H_0 x^{2+2\gamma})^* I^H_0 x^{2+2\gamma}$, $\L_\gamma^H$, and $-\partial_\omega^2$, we arrive at the following result.
\begin{theorem}\label{thm:FuncRelH}
    Fix $\gamma>-1$ and define $\D_\gamma^H \coloneqq (\L_\gamma^H)^{1/2} - \gamma- 1$\index{$ \D_\gamma^H$} via functional calculus, as well as $D_\omega \coloneqq \frac{1}{i} \partial_\omega$. Then, we have the functional relation
    \begin{equation}
	(I^H_0 x^{2+2\gamma})^* I^H_0 x^{2+2\gamma} = \frac{2^{2\gamma+2} \pi}{\D_\gamma^H+1} \frac{B( (\D_\gamma^H + D_\omega)/2 +1+\gamma, (\D_\gamma^H - D_\omega)/2 +1+\gamma)}{B((\D_\gamma^H + D_\omega)/2+1, (\D_\gamma^H - D_\omega)/2+1)},
	\label{eq:FuncRelH}
    \end{equation}	
    as bounded operators on $L^2(\Dm_H, x^{2\gamma+3}\d V_H)$.
    For $\gamma=0$, this reduces to
    \begin{equation}
	4\pi (\L_0^H)^{-1/2} = (I^H_0 x^{2})^* I^H_0 x^{2} = x^{-1} (I^H_0)^\sharp I^H_0 x^2 \qquad \text{on} \quad L^2(\Dm_H, x^3\d V_H).
	\label{eq:FuncRelH0}
    \end{equation}
\end{theorem}

\add{
\begin{remark}
    The intertwining properties in Proposition \ref{prop:interH} and the functional relation with differential operators in \eqref{eq:FuncRelH0} are Poincar\'e disk analogues of results on the Euclidean $\Rm^2$ (involving the Euclidean Laplacian, see  \cite[Lemma 2.1 and Theorem 3.1, respectively]{Helgason2010}) and on the Euclidean unit disk (with the Keldysh operator $\L_0$ defined Section \ref{sec:IDOs}, see \cite[Theorems 9 and 11 respectively]{Monard2019a}).
\end{remark}
}

\subsection{Sobolev mapping properties}
\label{sec:Sobolev_Mapping_Properties}

% mapping properties in Sobolev scales, C^\infty isomorphism property
For $s \ge 0$, define the space\footnote{The $w$ subscript is in reference to the fact that $\L_\gamma^H$ is a wedge operator.}
\begin{equation}
    H_w^{s,\gamma} (\Dm_H) \coloneqq \D \left((\L_{\gamma}^H)^{s/2}\right),
    \label{eq:Sobw}\index{$H_w^{s,\gamma} (\Dm_H)$}
\end{equation}
which can be obtained as the completion of $\Cev(\Dm_H)$ for the norm
\begin{equation}
    \|u\|_{H_w^{s,\gamma}}^2 = \sum_{n,k} (n+1+\gamma)^{2s} |u_{n,k}|^2, \qquad u = \sum_{n,k} u_{n,k} \widehat{\Phi^* Z_{n,k}^\gamma}.
    \label{eq:Hsnorm}
\end{equation}

\begin{theorem}\label{thm:isomorphismH} Fix $\gamma>-1$. Then:
    \begin{enumerate}[(a)]
	\item\label{item:isomorphismH1} for every $s\ge 0$, there exist $C_1, C_2$ such that for every $f\in \Cev(\Dm_H)$,
	    \begin{equation}
       \begin{aligned}
		&C_1 \|(I_0^H x^{2+2\gamma})^* I_0^Hx^{2+2\gamma}f \|_{H_w^{s+\min (1,1+\gamma),\gamma}}\\ &\le  \|f\|_{H_w^{s,\gamma}}\\& \le C_2 \|(I_0^H x^{2+2\gamma})^* I_0^H x^{2+2\gamma}f \|_{H_w^{s+\max(1,1+\gamma),\gamma}}.
      \end{aligned}
		\label{eq:tame}
	    \end{equation}

	\item\label{item:isomorphismH2} We have $\cap_{s\ge 0} H_w^{s,\gamma} (\Dm_H) = \Cev(\Dm_H)$.
	\item\label{item:isomorphismH3} The operator $(I_0^H x^{2+2\gamma})^* I_0^H x^{2+2\gamma}$ is an isomorphism of $\Cev(\Dm_H)$. 	
    \end{enumerate}
\end{theorem}

\begin{remark}
    For $\gamma=0$, since $H^{0,0}_w= L^2(\Dm_H,x^{3}\d V_H)=x^{-3/2}L^2(\Dm_H,\d V_H)$, \eqref{eq:tame} implies
    \begin{equation}\label{eq:stability_continuity}
	C_1 \|f\|_{x^{-3/2}L^2(\Dm_H,\d V_H)}\le \|(I_0^H x^{2})^* I_0^H x^{2}f \|_{H_w^{1,0}} \le C_2 \|f\|_{x^{-3/2}L^2(\Dm_H,\d V_H)}.
    \end{equation}
    By \cite[Theorem 1 and Propositions 3.3 and 4.4]{Eptaminitakis2022}, one has a similar estimate for the normal operator associated with $I_0^H$ in terms of weighted 0-Sobolev spaces of Mazzeo-Melrose (\cite{Mazzeo1987}), which are natural Sobolev spaces in hyperbolic geometry: For $f\in \dot{C}^\infty(\Dm_H)$ one has, for $\delta\in (-1/2,1/2)$ and $ s\geq 0$,
    \begin{equation}\label{eq:edge_mapping}
	C_1\|f\|_{x^\delta H_0^s(\Dm_H,\d V_H)}\leq \|(I_0^H)^\sharp I_0^H f\|_{x^\delta H_0^{s+1}(\Dm_H,\d V_H)}
	\leq C_2\|f\|_{x^\delta H_0^s(\Dm_H,\d V_H)} .
    \end{equation}
    Here, $\dot{C}^\infty(\Dm_H)$ denotes smooth functions on $\Dm_H$ vanishing with all derivatives to infinite order at $\partial \Dm_H$, and for $k\in \mathbb{N}_0$,
    \begin{equation}
    \begin{aligned}
	&x^\delta H^k_0(\Dm_H,\d V_H)\\ &\coloneqq
\big\{u\in H_{loc}^k(\Dm^{\circ}_H):(x\partial_x)^m(x \partial _{\omega})^\ell (x^{-\delta }u)\in L^2(\Dm_H,\d V_H), \ m+\ell \leq k\big\};
   \end{aligned}
    \end{equation}
    for $s\geq 0$, $x^\delta H^s_0(\Dm_H,\d V_H)$ is defined by interpolation.
    Since we have $(I_0^H)^\sharp I_0^H=x(I_0^Hx^2)^*(I_0^Hx^2)x^{-2}$,  \eqref{eq:edge_mapping} specializes for $s=0$ and $\delta'\in (-{5}/{2},-{3}/{2})$ to
    \begin{equation}\label{eq:edge_mapping_2}
	C_1\|f\|_{x^{\delta'} L^2(\Dm_H,\d V_H)}\leq \|(I_0^Hx^2)^*(I_0^Hx^2) f\|_{x^{\delta'+1} H_0^{1}(\Dm_H,\d V_H)}
	\leq C_2\|f\|_{x^{\delta'} L^2(\Dm_H,\d V_H)}.
    \end{equation}
    Thus, \eqref{eq:stability_continuity} can be viewed as an extension of \eqref{eq:edge_mapping_2} to the critical weight $\delta'=-3/2$.
\end{remark}

Theorem \ref{thm:isomorphismH} implies the following important result.

\begin{corollary}\label{cor:adj_onto}
    The map
    \begin{equation*}
	(I_0^H)^\sharp \colon \muh^2 \Calp (\Ghbar) \to x \Cev(\Dm_H)
    \end{equation*}
    is onto.
\end{corollary}

\subsection{Range characterization and moment conditions.} As a corollary of Theorem \ref{thm:SVDH} and Theorem \ref{thm:isomorphismH}\ref{item:isomorphismH2}, we immediately state without proof.
\begin{corollary}[Range characterization] \label{prop:rangeSobolevH}
    Fix $\gamma>-1$. For any $s\ge 0$, we have
    \begin{equation}
	I_0^H x^{2+2\gamma} \left( H_w^{s,\gamma}(\Dm_H) \right) = \left\{ \sum_{n\ge 0} \sum_{k=0}^n u_{n,k}\ \sigma_{n,k}^\gamma \psi_{n,k}^{\gamma,H},\ \sum_{n,k} (n+1+\gamma)^{2s} |u_{n,k}|^2 <\infty\right\}.
	\label{eq:rangeSobolevH}
    \end{equation}
    Moreover,
    \begin{equation}
    \begin{aligned}
	&I_0^H x^{2+2\gamma} \left( \Cev(\Dm_H) \right)  \\
   &= \left\{ \sum_{n\ge 0} \sum_{k=0}^n u_{n,k}\ \sigma_{n,k}^\gamma \psi_{n,k}^{\gamma,H},\ \sum_{n,k} (n+1+\gamma)^{2s} |u_{n,k}|^2 <\infty \quad \forall s\in \Nm_0 \right\}.\quad
	\label{eq:rangeSmooth}
   \end{aligned}
    \end{equation}
\end{corollary}

\begin{remark}
    The rapid decay condition in \eqref{eq:rangeSmooth} can also be formulated as follows: For each $N$ there exists $C_N$ such that
    \begin{equation}
	\max_{0\le k\le n}|u_{n,k}|\le C_N(1+n)^{-N}, \quad n\in \Nm_0.
    \end{equation}
\end{remark}

The above spectral characterization can also be understood in terms of {\em moment conditions}, most appropriately formulated in vertex coordinates, which are described in Section \ref{sec:vertex}. Given a function $u$ on $\Ghbar$, we denote by $u^{\rmv}$ the function $u$ pulled back to vertex coordinates. By virtue of Lemma \ref{lem:horo2vertex} below, this means
\begin{equation}
    u^\rmv(\omega,s) \coloneqq u\left(\omega + \frac{3\pi}{2} + 2\tan^{-1}s, \frac{-2s}{1-s^2}\right), \qquad (\omega,s)\in \add{\Rm/2\pi\Zm} \times (-1,1).
\end{equation}

\begin{definition}
    We say that a function $u$ on $\Gh$ satisfies the {\em hyperbolic moment conditions} if for any $m\in \Nm_0$ and any polynomial $p_m$ of degree $m$ and same parity as $m$ on $\Rm$, there exists a homogeneous polynomial $P_m$ of degree $m$ on $\Rm^2$ such that
    \begin{equation}
	\int_{-1}^1 p_m \left( \frac{-2s}{1+s^2} \right) u^{\rmv} (\omega,s) \frac{2\d s}{1+s^2} = P_m(\cos(\omega),\sin(\omega)).
	\label{eq:remoments}
    \end{equation}
\end{definition}

\begin{remark}\label{rem:moments}
    Classically, moment conditions are only stated for the more specific case $p_m (x) = x^m$ (see e.g., \cite{Berenstein1993a}), $m\ge 0$. That they are equivalent follows from the fact that a homogeneous polynomial $P_m$ can  also be seen as a homogeneous polynomial of degree $P_{m+2\ell}$ for any $\ell\ge 0$ upon multiplying it by $(\cos^2\omega + \sin^2\omega)^\ell$. More generally, \eqref{eq:remoments} holds for any polynomial $p_m$ as soon as it holds for a sequence of polynomials $\{q_m\}_{m\in \Nm_0}$ where, for any $m\in \Nm_0$, $q_m$ has degree exactly $m$ and the same parity as $m$.

    We explain in Appendix \ref{sec:HMC} how conditions \eqref{eq:remoments} are the same as those formulated in \cite{Berenstein1993a}.
\end{remark}

\begin{theorem}\label{thm:moments} Fix $\gamma>-1$, $s\ge 0$, and let $u\in L^2_+ (\Gh, \muh^{-2\gamma}\d\beta\d a)$. Then, $u= I_0^H x^{2\gamma+2} f$ for some $f\in H_w^{s,\gamma}(\Dm_H)$, resp. $f\in \Cev(\Dm_H)$, if and only if:
    \begin{itemize}
	\item[(i)] $u$ satisfies the hyperbolic moment conditions \eqref{eq:remoments}, and
	\item[(ii)] letting $P_m(\cos\omega,\sin\omega)$ be as in \eqref{eq:remoments}, relative to the polynomials $\{p_m^\gamma\}_{m\ge 0}$ defined in \eqref{eq:psinkgamma}, and decomposing $P_m(\cos\omega,\sin\omega) = \sum_{k = 0}^m M_{m,k} e^{i(m-2k)\omega}$, we have the condition
	    \begin{equation}
		\sum_{m\ge 0} \sum_{k=0}^m \frac{(m+1+\gamma)^{2s}}{(\sigma_{m,k}^\gamma)^2} |M_{m,k}|^2 <\infty,
		\label{eq:Hscond}
	\end{equation}
	resp., for each $N$ there exists $C_N$ such that
	\begin{equation}
	    \max_{0\le k\le m}|M_{m,k}|\le C_N(1+m)^{-N}, \quad m\in \Nm_0.
	    \label{eq:Cinfcond}
	\end{equation}
    \end{itemize}
\end{theorem}

Theorem \ref{thm:moments} is the extension to integrands with non-trivial boundary behavior of the characterization appearing in \cite[Theorem 3.10]{Berenstein1993a} in the case of compactly supported and Schwartz class integrands.

\subsection{The case \texorpdfstring{$\gamma=0$}{gamma=0}. Boundary operators and range characterizations.}  More can be said about Corollary \ref{prop:rangeSobolevH} in the case $\gamma = 0$: First, the spectral decay can be expressed in terms of natural anisotropic Sobolev scales; second, the infinite-dimensional co-kernel can be understood in terms of boundary operators, akin to those defined in \cite{Pestov2004,Monard2015a} in the context of simple Riemannian surfaces.

\smallskip

\noindent\textbf{Sobolev spaces on $\Ghbar$.} Upon defining $\Cal(\Ghbar)\coloneqq \Psihf^* C_\alpha^\infty(\partial_+ S\Dm_E)$, where the space $C_\alpha^\infty(\partial_+ S\Dm_E)$ is defined in \eqref{eq:CalphaE}, the spaces $\muh^2 \Cal(\Ghbar)$ and $\muh \Cal(\Ghbar)$ are two distinct core domains of $L^2 (\Ghbar, \d\beta\d a)$-essential-self-adjointness for $\T_0^H$. Accordingly, we denote by $\T_{0,D}^H$ (resp. $\T_{0,N}^H$) the graph closure of $(\T_0, \muh^2 \Cal(\Ghbar))$ (resp., of $(\T_0, \muh \Cal(\Ghbar))$), thought of as a Dirichlet (resp. Neumann) realization of $\T_0^H$ based on the Taylor expansions of elements in $\muh^2 \Cal(\Ghbar)$ (resp. $\muh \Cal(\Ghbar)$). The involution $\S_A^H$ induces the even/odd decompositions $L^2 = L^2_+ \stackrel{\perp}{\oplus} L^2_-$ and $\Cal = \Calp \oplus \Calm$, with projection maps $(id \pm (\S_A^H)^*)/2$. Moreover, since $\muh = \muh\circ\S_A^H$, we also have $\muh\Cal = \muh\Calp \oplus \muh\Calm$. Then, smoothness scales that  respect the symmetries and the boundary behavior of $I_0^H x^2$ will arise by defining, for $s\ge 0$,
\begin{equation}
\begin{aligned}
	&H_{T,D,\pm}^s (\Ghbar) \coloneqq (\T_{0,D}^H)^{-s/2} \left( L^2_\pm (\Ghbar, \d\beta\d a) \right), \\
	&H_{T,N,\pm}^s (\Ghbar) \coloneqq (\T_{0,N}^H)^{-s/2} \left( L^2_\pm (\Ghbar, \d\beta\d a) \right).
    \label{eq:SobTH}
    \end{aligned}
\end{equation}
Then, observing that $\T_0^H$ commutes with the projections $(id \pm (\S_A^H)^*)/2$, one can obtain the decompositions $\D((\T_{0,\bullet}^H)^{s/2}) = H_{T,\bullet,+}^s (\Ghbar) \oplus H_{T,\bullet,-}^s (\Ghbar)$ for $\bullet\in \{D,N\}$.

In Section \ref{sec:boundary_ops}, we define two ``boundary'' operators
\begin{equation}\label{eq:CPCinf}
    C_-^H \colon \muh^2 \Calp(\Ghbar) \to \muh^2 \Calp(\Ghbar), \qquad P_-^H \colon \muh \Calm(\Ghbar) \to \muh^2 \Calp(\Ghbar)
\end{equation}
which, for every $s\ge 0$, extend to bounded operators
\begin{equation}
	C_-^H \colon H_{T,D,+}^s (\Ghbar) \to H_{T,D,+}^s (\Ghbar), \qquad P_-^H \colon H_{T,N,-}^s (\Ghbar) \to H_{T,D,+}^s(\Ghbar),
	\label{eq:CP}
\end{equation}
and help characterize the range of $I_0^H x^2$ as follows.

\begin{theorem}\label{thm:gammazeroHalt}
    Let $u\in \muh^2 \Calp(\Ghbar)$ (resp. $H_{T,D,+}^{s+1/2}(\Ghbar)$ for $s\ge 0$ fixed).  The following are equivalent:
    \begin{enumerate}[(a)]
	\item \label{gammazeroHit0alt}There exists $f\in \Cev (\Dm_H)$ (resp. $\wtH^{s,0}_w (\Dm_H)$) such that $u = I_0^H x^2 f$;
	\item \label{gammazeroHit1alt} $C_-^H u = 0$;
	\item \label{gammazeroHit2alt} $u = P_-^H w$ for some $w\in \muh \Calm(\Ghbar)$ (resp. $w\in H_{T,N,-}^{s+1/2} (\Ghbar)$);
	\item \label{gammazeroHit3alt} $u$ satisfies the hyperbolic moment conditions \eqref{eq:remoments} (case $\gamma=0$) with the condition \eqref{eq:Cinfcond} (resp. \eqref{eq:Hscond}).
    \end{enumerate}
\end{theorem}

Aside from condition \ref{gammazeroHit3alt} discussed below Theorem \ref{thm:moments}, Theorem \ref{thm:gammazeroHalt} gives two more extensions of known results: Condition \ref{gammazeroHit2alt} is the analogue to Pestov-Uhlmann's range characterization of the X-ray transform on simple surfaces \cite[Thm 4.4.(ii)]{Pestov2004}, extended to the case of the hyperbolic space and to non-smooth classes, see also Theorem \ref{thm:rangeP} below; condition \ref{gammazeroHit1alt} is the generalization to hyperbolic space of \cite{Monard2015a}.

\smallskip

\noindent \textbf{Outline.} The remainder of the article is organized as follows.

We first discuss in Section \ref{sec:prelim} all building blocks needed for the main results. In Section \ref{sec:Eucl}, we discuss analogous results in the Euclidean disk, including a new Sobolev-space range characterization of the X-ray transform in Theorem \ref{thm:rangeP}. In Section \ref{ssec:models}, we discuss three parameterizations of hyperbolic geodesics (vertex, horocyclic, and cosphere bundle) and equivalences between them. We prove a Santal\'o-type formula (Proposition \ref{prop:SantaloH}) for the hyperbolic disk in Section \ref{sec:SantaloH}. Finally, we explicitly describe in Section \ref{sec:ProjEq} the projective equivalence between Euclidean and Poincar\'e disks, and the intertwining implications at the level of the X-ray transforms and backprojection operators (notably making using of Santal\'o's formula).

With these building blocks being set, Section \ref{sec:normalOps} provides brief proofs of the main results, Proposition \ref{prop:charac} through Theorem \ref{thm:moments}.

Finally, Section \ref{sec:boundary_ops} discusses the case $\gamma=0$ and the proof of Theorem \ref{thm:gammazeroHalt} (given in  Section\ \ref{ssec:gammazeroH}), discussing along the way a hyperbolic generalization of the incoming/outgoing boundaries of the unit tangent bundle of a simple surface (Section\ \ref{ssec:scattering_map}), boundary operators $C_{\pm}^H, P_{\pm}^H$ helping characterize the range of $I_0^H x^2$ (Section\ \ref{ssec:boundaryOps}), and an ambient description of the space $C_{\alpha,+}^\infty(\Ghbar)$ obtained through an appropriate gluing of the incoming and outgoing boundaries, usually kept disconnected in prior literature (see, e.g., \cite{Graham2019}), in Section\ \ref{ssec:smoothstruct}.

\section{Preliminaries} \label{sec:prelim}

\subsection{Known results on the Euclidean disk} \label{sec:Eucl}

The Euclidean disk $\Dm_E$ is equipped with its usual smooth structure, so that the function $d$ in \eqref{eq:d} is a smooth bdf. We parameterize geodesics through $\Dm_E$ via fan-beam coordinates (inward-pointing boundary): Given $e^{i\beta}\in \partial \Dm$ and an inward-pointing vector of direction $e^{i(\beta+\alpha+\pi)}$ for some $\alpha\in [-\pi/2,\pi/2]$, a unit-speed parameterization of the Euclidean geodesic passing through $(e^{i\beta}, e^{i(\beta+\pi+\alpha)})$ is given by
\begin{align}
\begin{split}
    \gamma_{\beta,\alpha}^E (u) &= e^{i\beta} + (u+\cos \alpha) e^{i(\beta+\pi+\alpha)} \\
	&= e^{i(\beta+\alpha+\pi)} (u + i\sin\alpha), \quad u\in [-\cos\alpha, \cos\alpha],
\end{split}
\label{eq:Egeo}\index{$ \gamma_{\beta,\alpha}^E$}
\end{align}
with endpoints $\gamma_{\beta,\alpha}^E (-\cos\alpha) = e^{i\beta}$ and $\gamma_{\beta,\alpha}^E (\cos\alpha) = e^{i (\beta+\pi+2\alpha)}$, attaining its vertex\footnote{The point closest to the origin.} when $u=0$. Then, the {\bf Euclidean X-ray transform} is defined by
\begin{equation}
    I_0^E f(\beta,\alpha) = \int_{-\cos\alpha}^{\cos\alpha} f(e^{i(\beta+\alpha+\pi)} (u + i\sin\alpha))\ du, \qquad (\beta,\alpha)\in \partial_+ S\Dm_E.
    \label{eq:EI0}
\end{equation}

\subsubsection{The space \texorpdfstring{$\Calp(\partial_+ S\Dm_E)$}{C alpha inifnity +}}

Recall that the {\bf scattering relation} is the map
\begin{equation}
    \mathcal{S}^E:\partial_\pm S\Dm_E\to \partial_\mp S\Dm_E,\quad (z,w)\mapsto \varphi_{\tau(z,\pm w)}(z,w),
\end{equation}
where $\tau(z,w)$ denotes the exit time for the orbit of the geodesic flow with initial data $(z,w) $ \add{and $\partial_- S\Dm_E=(\Rm/2\pi \mathbb{Z})_{\beta}\times [\pi/2,3\pi/2]_\alpha$ is the outward pointing boundary expressed in fan-beam coordinates, in which $(z,w) = (e^{i\beta}, e^{i(\beta+\pi+\alpha)})$. The scattering relation is expressed in fan beam coordinates as}
\begin{equation}
    \quad \S^E (\beta,\alpha) = (\beta+\pi+2\alpha, \pi-\alpha).
    \label{eq:EscatRel}
\end{equation}
Such a relation allows to extend a function on $\partial_+ S\Dm_E$ to $\partial S\Dm_E$ by evenness or oddness with respect to $\S^E$ through the operators
\begin{equation}
    A_\pm u (z,w) \coloneqq \left\{
    \begin{array}{rl}
	u(z,w), &\qquad (z,w)\in \partial_+ S\Dm_E \\
	\pm u(\S^E(z,w)), &\qquad (z,w)\in \partial_- S\Dm_E.
    \end{array}
    \right.
    \label{eq:ApmE}
\end{equation}
As mentioned in \cite[Section 4.3]{Monard2015a}, $A_\pm$ is bounded when viewed as an operator $L^2( \partial_+ S\Dm_E,\d\beta\d \alpha)\to L^2( \partial S\Dm_E,\d\beta\d \alpha) $ and as an operator $L^2( \partial_+ S\Dm_E,\mu \d\beta\d \alpha)\to L^2( \partial S\Dm_E,|\mu|\d\beta\d \alpha) $, whose adjoint in either functional setting is given by
\begin{equation}\label{eq:A_star}
    A_\pm ^*u (z,w)=u(z,w)\pm u(\mathcal{S}^E(z,w)), \quad (z,w)\in \partial _+ S\Dm_E.
\end{equation}
If $u\in C^\infty(\partial_+ S\Dm_E)$, then $A_{\pm} u$ is smooth away from $\partial_0 S\Dm_E \add{\coloneqq\partial_- S\Dm_E\cap \partial_+ S\Dm_E}$, though making it smooth on $\partial S\Dm_E$ requires constraining its Taylor expansion off of $\partial_0 S\Dm_E$. Related to this issue, an important subspace of $C^\infty (\partial_+ S\Dm_E)$ is defined (generally on a simple Riemannian surface) in \cite{Pestov2005} as
\begin{equation}
    \Cal(\partial_+ S\Dm_E) \coloneqq \{ u\in C^\infty(\partial_+ S\Dm_E), \ A_+ u\in C^\infty(\partial_+ S\Dm_E) \}.
    \label{eq:CalphaE}
\end{equation}
In particular, functions in $\Cal(\partial_+ S\Dm_E)$ can also be thought of as restrictions of functions that are smooth on $\partial S\Dm_E$ and invariant under $\S^E$. We now explain how this space helps capture the mapping properties of $I_0^E$. Since $I_0^E$ generally does not depend on geodesic orientation, we encode this further symmetry by defining
\begin{equation}
    \Calp(\partial_+ S\Dm_E) \coloneqq \{ u\in C_{\alpha}^\infty(\partial_+ S\Dm_E),\ u\circ \S_A^E = u \},
    \label{eq:CalphappE}
\end{equation}
where $\S_A^E\colon (\beta,\alpha)\mapsto (\beta+\pi+2\alpha, -\alpha)$ denotes the composition of the scattering relation and the antipodal map
\begin{equation}
    A_E:S\Dm_E\to S\Dm_E, \qquad (z,w)\mapsto (z,-w),
    \label{eq:AE}
\end{equation}
and encodes the orientation-reversing map of geodesics.

We now claim that for any $\gamma>-1$,
\begin{equation}
    I_0^E d^{\gamma} (C^\infty(\Dm_E)) \subset \mu^{2\gamma+1} \Calp(\partial_+ S\Dm_E).
    \label{eq:mappingI0E}
\end{equation}
Indeed, notice that with $\gamma_{\beta,\alpha}^E (u)$ defined in \eqref{eq:Egeo}, we have $d(\gamma_{\beta,\alpha}^E(u)) = \cos^2\alpha - u^2$. Then for $f\in C^\infty(\Dm_E)$, changing variable $u = t(\cos \alpha)$ for $t\in [-1,1]$ in the integral \eqref{eq:EI0}, we obtain
\begin{equation}
    I_0^E d^\gamma f (\beta,\alpha) = \mu^{2\gamma+1} \int_{-1}^1 (1-t^2)^\gamma f\left( e^{i (\beta+\pi+\alpha)} (t \cos \alpha + i \sin \alpha)\right)\ dt.
\end{equation}
This last integral is seen to be well-defined and smooth on $\partial S\Dm_E$, invariant under both $\S^E$ and $\S_A^E$. This exactly means that its restriction to $\partial_+ S\Dm_E$ belongs to $\Calp(\partial_+ S\Dm_E)$.

\begin{remark}
    Similar statements hold true on convex, nontrapping Riemannian manifolds, see e.g. \cite[Proposition 6.13]{Monard2021} or \cite[Theorem 2.1]{Mazzeo2021}.
\end{remark}

\subsubsection{Singular value decomposition}

We follow notation from \cite{Mishra2022}, replacing $I_0$ there by $I_0^E$ here. For any $\gamma>-1$, defining
\begin{equation*}
   L_+^2 (\partial_+ S\Dm_E, \mu^{-2\gamma}\d\beta\d\alpha) \coloneqq L^2 (\partial_+ S\Dm_E, \mu^{-2\gamma}\d\beta\d\alpha) \cap \ker ((\S_A^E)^*-id),
\end{equation*}
the operator
\begin{equation}
    I_0^E d^\gamma \colon L^2 (\Dm_E, d^\gamma \d V_E) \longrightarrow L_+^2 (\partial_+ S\Dm_E, \mu^{-2\gamma}\d\beta\d\alpha)
    \label{eq:I0dgamma}
\end{equation}
is bounded, with Hilbert space adjoint $(I_0^E)^\sharp \mu^{-2\gamma-1}$, where $(I_0^E)^\sharp $ denotes the Euclidean backprojection operator.

\begin{theoremlit}[Theorem 1 in \cite{Mishra2022}]\label{thm:SVD}
    \emph{Fix $\gamma>-1$ and let $\psi_{n,k}^{\gamma}$, $Z_{n,k}^\gamma $, and $(\sigma_{n,k}^\gamma)^{2}$ be as in \eqref{eq:psinkgamma}, \eqref{eq:Zernike}, and \eqref{eq:signkgammaH}, respectively.}
    \begin{enumerate}[(i)]
	\item The operator $(I_0^E)^\sharp \mu^{-2\gamma-1}\colon L_+^2(\partial_+ S\Dm_E, \mu^{-2\gamma} \d \beta\d\alpha)\to L^2(\Dm_E, d^\gamma\d V_E)$ has kernel
	    \begin{equation}
		\ker ((I_0^E)^\sharp \mu^{-2\gamma-1}) = \mathrm{span} \left( \psi_{n,k}^\gamma,\quad n\ge 0,\ k<0 \text{ or } k>n \right).
		\label{eq:keradjoint}
	    \end{equation}
	    \item The singular value decomposition of
		\begin{equation}
		    I_0^E d^\gamma\colon L^2(\Dm_E, d^\gamma \d V_E) \to L_+^2(\partial_+ S\Dm_E, \mu^{-2\gamma}\d \beta\d\alpha)/\ker((I_0^E)^\sharp \mu^{-2\gamma-1})	
		\end{equation}
		\emph{is given by $\left(\widehat{Z_{n,k}^\gamma}, \psi_{n,k}^\gamma, \sigma_{n,k}^\gamma\right)_{n\ge 0,\ 0\le k\le n}$.}
	\end{enumerate}
\end{theoremlit}

\subsubsection{Intertwining differential operators, functional relations, and \texorpdfstring{$C^\infty(\Dm_E)$}{C infinity}\break-mapping properties} \label{sec:IDOs}

From \cite{Mishra2022}, we have, for any $\gamma>-1$,
\begin{eqnarray*}
    I_0^E d^\gamma \circ \L_\gamma &=& \T_\gamma \circ I_0^E d^\gamma \qquad \text{on} \quad C^\infty(\Dm_E),
     \\
    \L_\gamma \circ (I_0^E)^\sharp \mu^{-2\gamma-1} &=& (I_0^E)^\sharp \mu^{-2\gamma-1} \circ \T_\gamma \qquad \text{ on } \quad  C^\infty( (\partial_+ S\Dm_E)^{\circ}),
\end{eqnarray*}
where
\begin{align}
    \L_\gamma &\coloneqq -(1-\rho^2)^{-\gamma} \rho^{-1}\partial_\rho\left((1-\rho^2)^{\gamma+1} \rho\partial_\rho\right)-\rho^{-2}\partial _\omega^2+(1+\gamma^2)id,\label{L_gamma_Eucl} \\
    \T_\gamma &\coloneqq -\mu^{2\gamma} T \mu^{-2\gamma} T + \gamma^2 id, \qquad T\coloneqq \partial_\beta-\partial_\alpha.\label{eq:T_gamma_Eucl}
\end{align}
The Zernike polynomials $\{Z_{n,k}^\gamma\}_{n,k}$ make up the full eigenvalue decomposition of the $L^2_\gamma$-essentially-self-adjoint operator $(\L_\gamma, C^\infty(\Dm_E))$, whose graph closure we denote by $\L_{\gamma,C^\infty}$, as $\L_\gamma Z_{n,k}^\gamma = (n+1+\gamma)^2 Z_{n,k}^\gamma$ for all $n,k$. One may then define the scale of Hilbert spaces
\begin{equation}
\begin{aligned}
    \wtH^{s,\gamma}(\Dm_E) &\coloneqq \D (\L_{\gamma,C^\infty}^{s/2})\\
     &= \left\{u = \sum_{n\ge 0} \sum_{k=0}^n u_{n,k} \widehat{Z_{n,k}^\gamma}, \qquad \sum_{n,k} (n+1+\gamma)^{2s} |u_{n,k}|^2 <\infty\right\}.
    \label{eq:Hsgamma}\index{$\wtH^{s,\gamma}(\Dm)$}
\end{aligned}
\end{equation}
Combined with Theorem \ref{thm:SVD}, we obtain a description of the range of $I_0^E d^\gamma$ when defined on this special Sobolev scale:
\begin{equation}
    I_0^E d^\gamma \left( \wtH^{s,\gamma}(\Dm_E) \right) = \left\{ \sum_{n\ge 0} \sum_{k=0}^n u_{n,k}\ \sigma_{n,k}^\gamma \psi_{n,k}^\gamma, \qquad \sum_{n,k} (n+1+\gamma)^{2s} |u_{n,k}|^2 <\infty\right\}.
    \label{eq:rangeSobolev}
\end{equation}

Further, large-$n$ estimates on $\sigma_{n,k}^\gamma$ in \cite[Lemma 12]{Mishra2022} show that for $s\ge 0$,
\begin{equation}
    \wtH^{s+\max (1,1+\gamma), \gamma}(\Dm_E) \subset (I_0^E)^\sharp \mu^{-2\gamma-1} I_0^E d^\gamma(\wtH^{s,\gamma}(\Dm_E))\subset \wtH^{s+\min (1,1+\gamma), \gamma}(\Dm_E).
    \label{eq:SobolevMapping}
\end{equation}
In particular, the normal operator $(I_0^E)^\sharp \mu^{-2\gamma-1} I_0^E d^\gamma$ is a tame map with tame inverse of $C^\infty(\Dm_E)$, when the latter is equipped with the graded family of semi-norms $\{\|\cdot\|_{\wtH^{p,\gamma}}\}_{p\in \Nm_0}$.

\begin{theoremlit}[Theorem 2 in \cite{Mishra2022}]\label{thm:Cinf}
    \emph{Fix $\gamma>-1$. The operator $(I_0^E)^\sharp \mu^{-2\gamma-1} I_0^E d^\gamma$ is an isomorphism of $C^\infty(\Dm_E)$.}
\end{theoremlit}

\subsubsection{The special case \texorpdfstring{$\gamma=0$}{gamma = 0}} \label{sec:gammazero}

In this case, the structure of the singular values simplifies, as $\sigma_{n,k}^0 = \frac{\sqrt{4\pi}}{\sqrt{n+1}}$ does not depend on $k$, allowing to refine some range descriptions. Below, we recall the results from \cite{Monard2019a}, but also give a Sobolev refinement, in Theorem \ref{thm:rangeP} below, of \cite[Thm 4.4.(ii)]{Pestov2004}. This requires defining new Hilbert scales on $\partial_+ S\Dm_E$, which generalize those previously defined in \cite{Monard2019a}.

\smallskip

\noindent \textbf{Sobolev scales on  $\partial_+ S\Dm_E$.} From the mapping property (see \cite[eqs. (25-26)]{Mishra2022})
\begin{equation}
    T(\Cal(\partial_+ S\Dm_E)) \subset \mu \Cal(\partial_+ S\Dm_E), \qquad T (\mu \Cal(\partial_+ S\Dm_E)) \subset \Cal(\partial_+ S\Dm_E),
    \label{eq:Tmaps}
\end{equation}
the operator $\T_0 \coloneqq -T^2$ can be equipped with domain $\mu\Cal(\partial_+ S\Dm_E)$ or\break  $\Cal(\partial_+ S\Dm_E)$, and $\T_0$ is $L^2(\partial_+ S\Dm_E,\d\beta\d\alpha)$-symmetric, and, in fact, essentially self-adjoint with respect to either of them. Indeed, the case of $\mu\Cal(\partial_+ S\Dm_E)$ is covered in \cite[Theorem 7, case $\gamma=0$]{Mishra2022}, and the case of $\Cal(\partial_+ S\Dm_E)$ can be proved using similar arguments. From the looks of their Taylor expansions off of $\partial_0 S\Dm_E$, these domains can be thought of as Dirichlet and Neumann core domains, respectively. Thus, we denote by $\T_{0,D}, \T_{0,N}$ the $L^2(\partial_+ S\Dm_E,\d\beta\d\alpha)$-graph closure of $(\T_0, \mu \Cal(\partial_+ S\Dm_E))$ and $(\T_0, \Cal(\partial_+ S\Dm_E))$, respectively. The involution $\S_A^E$ induces the even/odd decompositions $L^2 = L^2_+ \stackrel{\perp}{\oplus} L^2_-$ and $\Cal = \Calp \oplus \Calm$, with projection maps $(id \pm (\S_A^E)^*)/2$. Since $\mu \circ \S_A^E = \mu$, we also have $\mu \Cal = \mu \Calp \oplus \mu \Calm$. This allows us to define the following anisotropic Dirichlet and Neumann Sobolev scales for $s\ge 0$:
\begin{equation}
    H_{T,\bullet,\pm}^s (\partial_+ S\Dm_E) \coloneqq \T_{0,\bullet}^{-s/2} (L^2_\pm (\partial_+ S\Dm_E,\d\beta\d\alpha)), \qquad \bullet \in \{D,N\},
    \label{eq:SobT}
\end{equation}
so that $\D(\T_{0,\bullet}^{s/2}) = H_{T,\bullet,+}^s (\partial_+ S\Dm_E) \oplus H_{T,\bullet,-}^s (\partial_+ S\Dm_E)$ for $\bullet\in \{D,N\}$. The spaces $H_{T,D,+}^s$ coincide with the spaces $H_{T,+}^s(\partial_+ SM)$ originally introduced in \cite{Monard2019a}.

Let us define, for $n\ge 0$ and $k\in \Zm$, the $L^2(\partial_+ S\Dm_E, \d\beta\d\alpha)$-normalized functions
\begin{equation}
    \begin{split}
	\psi_{n,k} \coloneqq \frac{(-1)^n}{2\pi} e^{i(n-2k)(\beta+\alpha)} (e^{i(n+1)\alpha} + (-1)^n e^{-i(n+1)\alpha}), \\[0.2cm]
	\phi_{n,k} \coloneqq \frac{(-1)^n}{2\pi} e^{i(n-2k)(\beta+\alpha)} (e^{i(n+1)\alpha} - (-1)^n e^{-i(n+1)\alpha}).	
    \end{split}
    \label{eq:phipsiE}
\end{equation}
Up to scalar multiplication, the functions $\psi_{n,k}$ coincide with $\psi_{n,k}^0$ defined in \eqref{eq:psinkgamma}, and $\psi_{n,k}$ (resp. $\phi_{n,k}$) corresponds to $u'_{p,q}$ (resp. $v'_{p,q}$) defined in \cite{Mishra2019} upon reindexing $(p,q)\mapsto (n = -p+2q, k = q-p)$. These families are Hilbert bases of $L^2_+(\partial S\Dm_E, \d\beta\d\alpha)$ and $L^2_-(\partial S\Dm_E, \d\beta\d\alpha)$, respectively, and from \cite[Proposition 3.1]{Mishra2019} we also have the characterizations
\begin{align}
	\mu \Calp(\partial_+ S\Dm_E) = \Bigg\{ \sum_{n\ge 0}\sum_{k\in \Zm} u_{n,k}\ \psi_{n,k} \Bigg\}, \quad \Calm (\partial_+ S\Dm_E) = \Bigg\{ \sum_{n\ge 0}\sum_{k\in \Zm} u_{n,k}\ \phi_{n,k} \Bigg\},
    \label{eq:spans}
\end{align}
where the expansions have rapid decay, in the sense that for all $a,b\in \Nm$ one has $\sup_{n,k}\{ |u_{n,k}| (n+1)^{2a} (n-2k)^{2b} \}<\infty$. Noticing that $\psi_{n,k}, \phi_{n,k}$ are eigenvectors of $-T^2$, each with corresponding eigenvalue $(n+1)^2$, we deduce the following characterizations:
\begin{align}
    \begin{split}
	H_{T,D,+}^s (\partial_+ S\Dm_E) &= \Bigg\{ \sum_{n\ge 0}\sum_{k\in \Zm} u_{n,k}\ \psi_{n,k}, \quad \sum_{n,k} (n+1)^{2s} |u_{n,k}|^2 <\infty  \Bigg\}, \\[0.2cm]
	H_{T,N,-}^s (\partial_+ S\Dm_E) &= \Bigg\{ \sum_{n\ge 0}\sum_{k\in \Zm} v_{n,k}\ \phi_{n,k}, \quad \sum_{n,k} (n+1)^{2s} |v_{n,k}|^2 <\infty  \Bigg\}.	
    \end{split}
    \label{eq:HTcharac}
\end{align}

\smallskip

\noindent \textbf{Boundary operators and range characterizations.} Out of the scattering relation and fiberwise Hilbert transform, one can construct a `boundary' operator
\begin{align*}
	C_- \colon \mu \Calp(\partial_+ S\Dm_E) \to \mu \Calp(\partial_+ S\Dm_E),
\end{align*}
that directly captures the infinite-dimensional co-kernel of $I_0^E$. Specifically,
\begin{equation}
    C_- \coloneqq \frac{1}{2} A_-^* H_- A_-,
    \label{eq:Cminus}
\end{equation}
where $A_-, A_-^*$ are defined in \eqref{eq:ApmE} and \eqref{eq:A_star} and $H$ is the odd part of the fiberwise Hilbert transform on $\partial S\Dm_E$, whose integral expression is given by \eqref{eq:Hminus}.
Explicitly, one finds that
\begin{equation}\label{eq:Cminus_spectral}
    C_-\psi_{n,k} = i (1_{k<0} - 1_{k>n}) \psi_{n,k}, \qquad n\ge 0,\ k\in \Zm;
\end{equation}
in particular, $C_- \colon H_{T,D,+}^s(\partial_+ S\Dm_E)\to H_{T,D,+}^s(\partial_+ S\Dm_E)$ is bounded for every $s\ge 0$.

\begin{theoremlit}[Theorem 6 in \cite{Monard2019a}, case $M = (\Dm_E,e)$] \label{thm:rangegammazero}
    \emph{We have}
    \begin{equation}
	\begin{split}
	    I_0^E (\wtH^{s,0}(\Dm_E)) &= H_{T,D,+}^{s+1/2}(\partial_+ S\Dm_E) \cap \ker C_-, \quad s\ge 0, \\[0.2cm]
	    I_0^E (C^\infty(\Dm_E)) &= \mu \Calp(\partial_+ S\Dm_E) \cap \ker C_-.
	\end{split}
	\label{eq:rangegammazero}
    \end{equation}	
   \emph{ Moreover, for all $s\ge 0$ and $f\in \wtH^{s,0}(\Dm_E)$, $\|f\|_{\wtH^{s,0}(\Dm_E)} = \frac{1}{\sqrt{4\pi}} \|I_0^E f\|_{H_{T,D,+}^{s+1/2}(\partial_+ S\Dm_E)}$.}
\end{theoremlit}

We now discuss a Sobolev scale generalization of the Pestov-Uhlmann range characterization \cite[Thm 4.4.(ii)]{Pestov2004} in terms of the boundary operator
\begin{equation}
    P_- = A_-^* H_- A_+,
    \label{eq:Pminus}
\end{equation}
defined analogously to \eqref{eq:Cminus}. In this case, we have that $P_- (\Calm(\partial_+ S\Dm_E)) \subset \mu \Calp (\partial_+ S\Dm_E)$,
$P_-$ extends by continuity as an operator $P_-\colon L^2_-(\partial_+ S\Dm_E)\to L^2_+ (\partial_+ S\Dm_E)$ and, reformulating \cite[Theorem 3.3]{Mishra2019} in the current notation, $P_-$ has singular value decomposition
\begin{equation}
    P_- \phi_{n,k} = -2i\ 1_{0\le k\le n}\ \psi_{n,k}, \qquad n\ge 0,\quad k\in \Zm,
    \label{eq:SVDP}
\end{equation}
In particular, it becomes immediate that $P_- \colon H_{T,N,-}^{s+1/2}(\partial_+ S\Dm_E)\to H_{T,D,+}^{s+1/2} (\partial_+ S\Dm_E)$ is bounded for any $s\ge 0$, and its range coincides with $I_0^E (\wtH^{s,0} (\Dm_E))$, which we state here without proof.

\begin{theorem} \label{thm:rangeP}
    For any $s\ge 0$, one has
    \begin{equation}
	I_0^E (\wtH^{s,0}(\Dm_E)) = P_- \left( H_{T,N,-}^{s+1/2}(\partial_+ S\Dm_E) \right) ,
    \end{equation}
    as subspaces of $H_{T,D,+}^{s+1/2}(\partial_+ S\Dm_E)$.
\end{theorem}

 This is a Sobolev generalization of \cite[Thm 4.4.(ii)]{Pestov2004} and the slight refinement \cite[Proposition A.2]{Mishra2019}, which provides the smooth counterpart
\begin{equation}\label{eq:PU_Smooth}
    I_0^E (C^\infty(\Dm_E)) = P_- \left( \Calm (\partial_+ S\Dm_E) \right).
\end{equation}

\begin{remark}[On notation comparison with prior literature]\label{rem:notation} In prior articles \cite{Mishra2019,Mishra2022}, for $\sigma_1, \sigma_2 \in \{+,-\}$, the following spaces were introduced:
	\begin{equation}
		C_{\alpha,\sigma_1,\sigma_2}^{\infty}(\partial_+ S\Dm_E) = \{ u\in C^\infty(\partial_+ S\Dm_E),\ A_{\sigma_1}u \in C^\infty(\partial S\Dm_E), \quad u\circ \S_A^E = \sigma_2 u\}.
	\end{equation}
	In light of the observation that $C_{\alpha,-,\sigma_2}^\infty (\partial_+ S\Dm_E) = \mu C_{\alpha,+,\sigma_2}^\infty (\partial_+ S\Dm_E)$ for $\sigma_2\in \{+,-\}$ and in an attempt to save notation, we have chosen in the present article to remove $\sigma_1$ from the notation. Hence, the correspondence (spaces on the right are those in the present article) is as follows: For $\sigma\in \{+,-\}$,
	\begin{equation}
		C^\infty_{\alpha,+} \leftarrow C_\alpha^\infty, \qquad C_{\alpha,-}^\infty \leftarrow \mu C_\alpha^\infty, \qquad C_{\alpha,+,\sigma}^\infty \leftarrow C_{\alpha,\sigma}^\infty, \qquad C_{\alpha,-,\sigma}^\infty \leftarrow \mu C_{\alpha,\sigma}^\infty.
	\end{equation}	
\end{remark}

\subsection{Models of hyperbolic geodesics} \label{ssec:models}

Recall that the unit sphere bundle $S\Dm_H^\circ$ in $(\Dm_{H}^\circ,g_H)$ is defined in \eqref{eq:SDH}. In what follows, we will sometimes use the identification
\begin{equation}\label{eq:identification_sphere_bundle}
    (z,w)=\big(z,{2\Re(c\ e^{i\theta}\partial_{{z}})}\big)\in S\Dm_H^{\circ} \quad \leftrightarrow \quad (z,\theta)\in \Dm_H^{\circ}\times (\Rm/2\pi \mathbb{Z}). \index{$\theta$}
\end{equation}
A distinguished $g_H$-unit speed geodesic passing through the $g_H$-unit vector $(z,\theta) = (0,0)$ is given by the function $\x(t)$ defined in \eqref{eq:hyp_fan_beam} for which $(\x(0),\dot \x(0)) = (0,1/2)$. Any other geodesic on $\Dm_H$ can be obtained as the image of $\x(t)$ under an element of $\mathrm{PSU}(1,1)$: Specifically, the unit-speed geodesic passing through $(z_1, \theta)$ at $t=0$ is given by
\begin{equation}
    \gamma_{z_1,\theta}(t) \coloneqq T_{z_1,\theta} (\x(t)), \qquad T_{z_1,\theta} (z) \coloneqq \frac{e^{i\theta}z + z_1}{1+e^{i\theta} \overline{z_1} z}.
    \label{eq:geodesic}
\end{equation}

We now discuss certain models of all oriented geodesics on $(\Dm_H,g_H)$.

\subsubsection{Vertex parameterization} \label{sec:vertex}

Using \eqref{eq:geodesic}, define, for $(\omega,s)\in \add{(\Rm/2\pi \mathbb{Z})}\times (-1,1)$,
\begin{equation}
    \gammav_{\omega,s}(t) \coloneqq T_{s e^{i\omega}, \omega+\pi/2} (\x(t)) = e^{i\omega} \frac{s+i\x(t)}{1+is \x(t)}, \qquad t\in \Rm,
    \label{eq:Hgeo}
\end{equation}
with $\x(t)$ defined in \eqref{eq:hyp_fan_beam}. This geodesic is parameterized so as to attain its vertex (point closest to the origin) at $t=0$, equal to $s e^{i\omega}$, with direction $\omega+\pi/2$. In this way the set $\add{(\Rm/2\pi \Zm)_{\omega}}\times (-1,1)_s$ gives all oriented geodesics, with orientation-reversing involution $(\omega,s)\mapsto (\omega+\pi, -s)$. Note the important identity used below, with $x$ as in \eqref{eq:xbdf}:
\begin{equation}
    x(\gammav_{\omega,s}(t)) = \frac{1-s^2}{1+s^2} \frac{1-\x(t)^2}{1+\x(t)^2} = \frac{1-s^2}{1+s^2} \frac{1}{\cosh(t)}.
    \label{eq:xvertex}
\end{equation}

\subsubsection{Horocyclic parameterization} \label{sec:horocycle}

We now discuss the derivation of the parameterization \eqref{eq:hyp_fan_beam}, whose analogy with Euclidean fan-beam coordinates will become clear in Proposition \ref{prop:projEquiv}.

The group $G=\mathrm{SL}(2,\Rm)$ acts on the upper half plane model of hyperbolic space
\begin{equation*}
	\Hm^2=\{z=x_\Hm+iy_\Hm\in \mathbb{C}:y_\Hm>0\},\quad g_{\Hm}=\frac{\d x_\Hm^2+\d y_\Hm^2}{y_\Hm^2}
\end{equation*}
by isometries via M\"obius transformations, that is, $\left[\begin{smallmatrix} a & b\\ c & d \end{smallmatrix}\right]\cdot z=\frac{az+b}{cz+d}$. An Iwasawa decomposition of $G$ reads $G = KNA$, where $A,N,K\leq G$ are the subgroups
\begin{eqnarray*}
	A&=&\left\{A(t)=\left[\begin{smallmatrix}
		e^{-t/2} & 0\\
		0 & e^{t/2}
	\end{smallmatrix}\right]:t\in \Rm\right\},    \quad
	N=\left\{N(a)=\left[\begin{smallmatrix}
		1 & a\\
		0 & 1
	\end{smallmatrix}\right]:a\in \Rm\right\},  \\
	K&=&\left\{K(\beta)=\left[\begin{smallmatrix}
		\cos(\beta/2) & \sin(\beta/2)\\
		-\sin(\beta/2) & \cos(\beta/2)
	\end{smallmatrix}\right]:\beta\in \Rm\right\}.
\end{eqnarray*}
Conjugating by the Cayley transform $C(z)=\frac{z-i}{z+i}$, which is an isometry $C:(\mathbb{H}^2,g_{\mathbb{H}})\to (\Dm_H^\circ,g_H)$, we obtain an action of $\mathrm{SL}(2,\Rm)$ on $\Dm_H^\circ$ by isometries. For $z\in \Dm_H^\circ$, we then have $K(\beta)\cdot z=e^{i\beta}z$, $A(t)\cdot z=\frac{z-\x(t)}{1-z\x(t)}$ (with $\x(t)$ defined in \eqref{eq:hyp_fan_beam}), and $N(a)\cdot z=\frac{z(2+ia)-ia}{ia z+(2-ia)}$. Further, $A(t)\cdot 0$ is a geodesic, and one finds that the expression \eqref{eq:hyp_fan_beam} of $\gamma_{\beta,a}(t)$ is nothing but
\begin{equation*}
    \gamma_{\beta,a}(t) = K(\beta)N(a)A(t)\cdot 0.
\end{equation*}
For future reference, $\gamma_{\beta,a}$ has endpoints
\begin{equation}\label{eq:endpoints}
    \gamma_{\beta,a}(-\infty)=e^{i\beta}\quad \text{ and }\quad\gamma_{\beta,a}(+\infty)= -e^{i\beta} \frac{1+ia}{1-ia} = e^{i(\beta+\pi + 2\tan^{-1}a)}.
\end{equation}
The parameterization \eqref{eq:hyp_fan_beam} is called {\em horocyclic} in that, when $\beta,t$ are frozen, the curve $a\mapsto K(\beta)N(a)A(t)\cdot 0$ parameterizes a horocycle with ideal point $e^{i\beta}$, generally with non-unit speed.

\smallskip

\noindent \textbf{The map $\pih$.} We now give the expression of the map $\pih$ defined in \eqref{eq:F}, used to define the backprojection operator \eqref{eq:backproj}.

\begin{lemma}\label{lem:pih}
    With $\pih$ defined in \eqref{eq:F}, we have that for every $(z,\theta)\in S\Dm_H^\circ$,
    \begin{equation}
	    \pih(z,\theta) = (\beta,a), \quad \text{where} \quad \beta = \theta + \pi + 2\arg (1- ze^{-i\theta}), \quad a = \frac{2\Im(z e^{-i\theta})}{1-|z|^2}.
	\label{eq:pihcomponents}
    \end{equation}
\end{lemma}

\begin{proof}
    By definition, $\pih(z,\theta)$ is the unique $(\beta,a)\in \Gh$ with $(\gamma_{\beta,a} (t),\dot\gamma_{\beta,a} (t))=(z, 2\Re(c\ e^{i\theta}\partial_{{z}}))$ for some $t\in \Rm$. Note that such a geodesic has equation $T_{z,\theta}(\x(t))$, where $T_{z,\theta}(z') = \frac{e^{i\theta}z' + z}{1+e^{i\theta}\overline{z}z'} = e^{i\theta} \frac{z'+\zeta}{\overline{\zeta}z' + 1}$ with $\zeta \coloneqq z e^{-i\theta}$\index{$\zeta$}. In particular, with $\x(\pm\infty) = \pm 1$, its points at infinity are given by
\begin{equation}
    T_{z,\theta} (-1) = e^{i\theta} \frac{-1+\zeta}{-\overline{\zeta}+1} = e^{i(\theta+\pi)} \frac{1-\zeta}{1-\overline{\zeta}}, \qquad T_{z,\theta}(1) = e^{i\theta} \frac{1+\zeta}{1+\overline{\zeta}}.
\end{equation}
To find the horocyclic coordinates $(\beta,a)$, we compare the above with $\gamma_{\beta,a}(\pm \infty)$ in \eqref{eq:endpoints}, giving
\begin{equation}
	e^{i(\theta+\pi)} \frac{1-\zeta}{1-\overline{\zeta}} = e^{i\beta}, \qquad e^{i\theta} \frac{1+\zeta}{1+\overline{\zeta}} = -e^{i\beta} \frac{1+ia}{1-ia}.
\end{equation}
The first equation above gives the $\beta$ component of \eqref{eq:pihcomponents} directly. Next, taking the ratio of the above, we obtain
\begin{equation}
	\frac{1+ia}{1-ia} = \frac{(1+\zeta)(1-\overline{\zeta})}{(1+\overline{\zeta})(1-\zeta)} = \frac{1+i \left( (\zeta-\overline{\zeta})/(i(1-|\zeta|^2)) \right)}{1-i \left( (\zeta-\overline{\zeta})/(i(1-|\zeta|^2)) \right)}.
\end{equation}
By injectivity of the map $\Rm\ni x\mapsto \frac{1+ix}{1-ix}$, $a = \frac{\zeta-\overline{\zeta}}{i(1-|\zeta|^2)}$, which then gives the second component of \eqref{eq:pihcomponents}. Lemma \ref{lem:pih} is proved.
\end{proof}

The equivalence between vertex and horocyclic parameterizations is as follows.
\begin{lemma}\label{lem:horo2vertex}
    For any $a\in \Rm$, $\beta\in \add{\Rm/2\pi\Zm}$, let $t_0 = -\log \sqrt{1+a^2}$. Then, we have
    \begin{equation}
    \begin{aligned}
	\gamma_{\beta,a} (t+t_0) &= \gammav_{\beta + \pi/2 -2 \tan^{-1}s, s} (t),\quad t\in \Rm,\\   \quad s &\coloneqq \frac{-a}{\sqrt{1+a^2}+1} = - \tan \left(\frac{1}{2} \tan^{-1}a \right).
	\label{eq:horo2vertex}\index{$s$}
   \end{aligned}
    \end{equation}	
\end{lemma}

\begin{proof}
    The horocyclic coordinates $(\beta,a)$ corresponding to the geodesic $\gamma^{\mathrm{v}}_{\omega,s}$ can be found by setting
    \begin{equation}
    \begin{aligned}
    \label{eq:b_a_to_s_omega}
	(\beta,a)&=\pih(se^{i\omega},\omega+\pi/2)\\
   &=\Big(\omega+\frac{3\pi}{2}+2\arg(1+is),\frac{2\Im(-is)}{1-s^2}\Big)=\Big(\omega+\frac{3\pi}{2}+2\tan^{-1}s,\frac{-2s}{1-s^2}\Big).	
   \end{aligned}
    \end{equation}
    Now, solve for $(\omega,s)$ using that $s\in (-1,1)$ to find $\omega=\beta+\pi/2-2\tan^{-1}s$ and $s=\frac{-a}{\sqrt{1+a^2}+1}$.

    If $(\beta,a)=\pih(se^{i\omega},\omega+\pi/2)$, $\gamma_{\beta,a}(t)$ and $ \gamma^{\mathrm{v}}_{\omega,s}(t)$ are unit speed parameterizations of the same oriented geodesic, hence there exists a unique $t'=t'(\beta,a)\in \Rm$ such that
    $\gamma_{\beta,a}(t'+t)=\gamma^{\mathrm{v}}_{\omega,s}(t)$.
    Setting $t=0$, we observe that $t'=\mathrm{argmin}_{t\in \Rm} |\gamma_{\beta,a}(t)|^2$. We now compute
    \begin{equation}\label{eq:mod_gamma}
	|\gamma_{\beta,a}(t)|^2=1+4\frac{\tanh^2(t/2)-1}{a^2(\tanh(t/2)+1)^2+4}=1-\frac{4}{e^{-t}+2+(1+a^2)e^t}.
    \end{equation}
    This is minimized when $e^{-t}+2+(1+a^2)e^t$ is minimized, which by differentiation occurs at $t'$ where $-e^{-t'}+(1+a^2)e^{t'}=0$, i.e., at $t'=t_0$ as defined in the \break statement.
\end{proof}

\subsubsection{Cosphere bundle parameterization}\label{sec:cosphere}

Another parameterization of the space of oriented geodesics on $\Dm_H^\circ$ appears in recent studies of the X-ray transform on asymptotically hyperbolic manifolds \cite{Graham2019} and is obtained by studying geodesic flows on the cosphere bundle. The purpose of this section is to show that this parameterization recovers, up to the sign, the horocycle parameterization \eqref{eq:hyp_fan_beam} in the case of the hyperbolic disk. See Lemma \ref{lm:cosphere-horocycle} below.

To recall the context, let $M$ be a $n+1$-dimensional compact manifold with boundary whose interior $M^\circ$ is equipped with an asymptotically hyperbolic metric in the sense that there exist local coordinates $(\tx,y)$ near $\partial M$, where $\tx$\index{$\tx$}\index{$y$} is a suitable bdf such that the metric takes the form
\begin{equation*}
	g = \frac{\d \tx^2 + \sum_{ij}h_{ij}(\tx)\,\d y^i\,\d y^j}{\tx^2};
\end{equation*}
a bdf $\tilde{x}$ with this property is called a \emph{geodesic boundary defining function}. In this context, the cosphere bundle $S^* M^\circ$ can be viewed as a subset of the $b$-cotangent bundle $^bT^*M$ (see \cite{Melrose1993}) restricted to $M^\circ$. That is,
\begin{equation*}
	S^* M^\circ = \left\{\left(\overline{\xi_0}\frac{\d\tx}{\tx} + \sum_i{\eta_i\,\d y^i}\right)\Big|_p\,:\,p\in M^\circ,\quad \overline{\xi_0}^2 + \tx^2\sum_{ij}h^{ij}\eta_i\eta_j = 1\right\}.
\end{equation*}
Then $S^*M^\circ$ extends smoothly as a submanifold $\overline{S^*M^\circ}$ of $^bT^* M$ to the boundary, with $\overline{S^*M^\circ}$ at the boundary given by
\begin{equation*}
    \partial\overline{S^*M^\circ} = \overline{S^*M^\circ}|_{\partial M} = \left\{\left(\overline{\xi_0}\frac{\d\tx}{\tx} + \sum_i{\eta_i\,\d y^i}\right)\Big|_p\,:\,p\in \partial M,\quad \overline{\xi_0}^2 = 1\right\}.
\end{equation*}
We have $\partial\overline{S^*M^\circ}= \partial_+S^*M^\circ \cup \partial_-S^*M^\circ$, where
\begin{equation*}
    \partial_{\pm} S^*M^\circ \coloneqq \left\{\left(\pm\frac{\d\tx}{\tx} + \sum_i{\eta_i\,\d y^i}\right)\Big|_p\,:\,p\in\partial M,\quad \sum_i{\eta_i\,\d y^i}\Big|_p\in T_p^* \partial M \right\}.
\end{equation*}
Then, the Hamiltonian vector field $X$\index{$X$} corresponding to geodesic flow can be written as $\tx\overline{X}$ where $\overline{X}$ is smooth and transverse to $\partial\overline{S^*M^\circ}$. Since the orbits of $X$ and $\overline{X}$ agree up to reparameterization on $S^*M^\circ$, we can identify geodesics with their incoming/outgoing b-covectors in $\partial_{\pm} S^*M^\circ$. Concretely, given a unit-speed geodesic $\gamma(t)$, we consider the limits
\[\lim_{t\to\mp\infty} g\left(\dot\gamma(t)\right)\]
where we interpret $g$ acting on a vector as converting it to a covector; such limits exist in $^bT^*M$ and belong to $\partial_{\pm} S^* M^\circ$.

In the case of $(\Dm_H^\circ,g_H)$, a family of geodesic bdfs, indexed by a constant $C>0$, is given by
\begin{equation*}
	\tx(z) = C\frac{1-|z|}{1+|z|}, \quad z\in \Dm_H^\circ,
	% \label{eq:xgbdf}
\end{equation*}
in which case $g_H = \frac{1}{\tx^2} \left( \d\tx^2 + \left(\frac{C^2-\tx^2}{2C}\right)^2\,\d\omega^2 \right)$. Then, writing $\gamma(t) = (\tx(t),\omega(t))$, we have
\begin{equation*}
    g_H(\dot \gamma(t)) = \frac{\dot{\tx}(t)}{\tx (t)} \frac{\d\tx}{\tx} + \left( \frac{C^2-\tx^2(t)}{2C} \right)^2 \frac{\dot\omega(t)}{\tx^2(t)} \d\omega.	
    % \label{eq:covector}
\end{equation*}

\begin{lemma}\label{lm:cosphere-horocycle}
    In the cosphere bundle parameterization, the geodesic $\gamma_{\beta,a}(t)$ defined in \eqref{eq:hyp_fan_beam} satisfies
    \begin{equation}
	\lim_{t\to\pm\infty}\frac{\dot{\tx}(t)}{\tx(t)} = \mp 1 \quad \text{and} \quad  \left( \frac{C^2-\tx^2(t)}{2C} \right)^2 \frac{\dot\omega(t)}{\tx^2(t)} = -a\ \text{ for all }t	,
    \end{equation}
    where we write $\gamma_{\beta,a}(t)= (\tx(t),\omega(t))$. In particular, the left-hand side of the second equation is constant. Consequently,
    \[   \lim_{t\to \pm\infty} g_H(\dot\gamma_{\beta,a}(t)) = \left( \mp\frac{\d\tx}{\tx} - a\,\d\omega\right)\big|_{\gamma_{\beta,a}(\pm \infty)} %\quad \text{and}\quad \lim_{t\to \infty} g_H(\dot\gamma_{\beta,a}(t)) = \left(-\frac{\d\tx}{\tx} - a\,\d\omega\right)\big|_{e^{i(\beta+\pi+2\tan^{-1}a)}},
\]
    so in particular the geodesic whose incoming covector is $\left(\frac{\d\tx}{\tx} + \eta\,\d\omega\right)|_{e^{i\beta}}$ is given by $\gamma_{\beta,-\eta}$.
\end{lemma}

Note that the measure $\d\beta \d a$ on $\Gh$ coincides with that induced by the symplectic form $\d\beta\wedge \d\eta$ on $T^*\Sm^1_{\beta,\eta}$.

\begin{proof} Write $r(t) \coloneqq |\gamma_{\beta,a}(t)|$. By direct calculation,
    \begin{equation}
	r^2 = \frac{e^{-t} -2 + (1+a^2)e^t}{e^{-t} + 2 + (1+a^2) e^t}, \qquad \frac{\d}{\d t} r^2 = 4\frac{-e^{-t}+(1+a^2) e^t}{(e^{-t} + 2 + (1+a^2)e^{t})^2}.
	\label{eq:direct}
    \end{equation}
    For the $\d\tx$ component, using \eqref{eq:direct}, we compute
    \begin{equation}
	\frac{\dot\tx}{\tx} = \frac{-2 \dot r}{1-r^2} = \frac{-1}{r} \frac{1}{1-r^2} \frac{\d}{\d t} r^2 = \frac{-1}{r} \frac{-e^{-t}+ (1+a^2) e^t}{e^{-t} + 2+ (1+a^2)e^t} \to \mp 1 \quad \text{as}\quad t\to \pm \infty.
    \end{equation}
    On to the $\d\omega$ component, from $1-r^2 = \frac{4}{e^{-t}+2+(1+a^2)e^t}$, we have $e^{-t}+2+(1+a^2)e^t = \frac{4}{1-r^2}$, thus $e^{-t}-2+(1+a^2)e^t = \frac{4r^2}{1-r^2}$. Moreover, from $\gamma_{\beta,a}(t) = r(t)e^{i\omega(t)}$, such that $\frac{\dot\gamma_{\beta,a}(t)}{\gamma_{\beta,a}(t)} = \frac{\dot r(t)}{r(t)} + i\dot\omega(t)$, we deduce that
    \[\dot\omega = \Im\left(\frac{\dot\gamma_{\beta,a}(t)}{\gamma_{\beta,a}(t)}\right) = \frac{-4a}{(e^{-t}-(1+a^2)e^t)^2+4a^2}.\]
    Note that
    \begin{align*}
	(e^{-t}-(1+a^2)e^t)^2+4a^2 &= e^{-2t}-2(1+a^2)+(1+a^2)e^{2t}+4a^2 \\
	&= e^{-2t}+2(1+a^2)+(1+a^2)e^{2t} - 4 \\
	&= (e^{-t}+(1+a^2)e^t)^2-2^2 \\
	&= (e^{-t}+2+(1+a^2)e^t)(e^{-t}-2+(1+a^2)e^t) \\
	&= \frac{4}{1-r^2}\cdot\frac{4r^2}{1-r^2} = \frac{16r^2}{(1-r^2)^2}.
    \end{align*}
    It follows that
    \[\left(\frac{C^2-\tx^2}{2C}\right)^2\dot\omega = \left(\frac{2r}{1-r^2}\right)^2 \frac{-4a}{16r^2/(1-r^2)^2} = -a, \]
    as claimed.
\end{proof}

\subsection{Santal\'o's formula} \label{sec:SantaloH}

Any smooth vector field on $S\Dm_H^\circ$ (and, more generally, on the sphere bundle of any oriented Riemannian surface) can be written as a linear combination over $C^\infty(\Dm_H^\circ)$ of three globally defined vector fields $X$, $X_\perp$, and $V$ (see \cite{Paternain2023}), acting on any $f\in C^\infty (S\Dm_H^\circ)$ by
\begin{equation}
    \begin{aligned}
        Xf(z,w)&=\frac{\d}{\d t}f(\phi_t(z,w))\big|_{t=0},\\
         X_\perp f{(z,w)}&=\frac{\d}{\d t}f(\psi_t(z,w))\big|_{t=0},\quad
    Vf{(z,w)}=\frac{\d}{\d t}f(r_t(z,w))\big|_{t=0}.
    \end{aligned}
\end{equation}
Here, $\phi_t$ denotes the geodesic flow, $\psi_t(z,w)$ is the curve in $S\Dm_H^\circ$ arising as the parallel transport of $(z,w)$ along the geodesic with initial velocity $(z,-w^\perp)$ (where $(\cdot)^\perp$ denotes fiberwise counterclockwise rotation by angle $\pi/2$), and $r_t(z,w)$ is the counterclockwise rotation of $(z,w)$ by angle $t$.
In terms of coordinates $(z,\theta)$, see \eqref{eq:identification_sphere_bundle}, they take the form (see \cite[Lemma 3.5.6]{Paternain2023})
\begin{equation}\label{SM_Vector_fields}
    \begin{aligned}
    X&= c \Big(2\Re(e^{i\theta}\partial_z)+2\Im(e^{i\theta}\partial_z \log c)\partial _\theta\Big),\\
    X_\perp&= c \left(2\Im(e^{i\theta}\partial_z)-2\Re (e^{i\theta}\partial_z \log c)\partial_\theta\right),\qquad
    V=\partial_\theta,
\end{aligned}
\end{equation}
with $c$ as in \eqref{eq:Poincare_metric}. In coordinates $(z,\theta)$, the Liouville form on $\Dm_H^{\circ}\times (\Rm/2\pi \mathbb{Z})$ is given by
\begin{equation}\label{eq:Liouville}
    \d\Sigma^3(z,\theta) = \d V_H \wedge \d\theta = c^{-2}(z) \frac{\d\bar{z}\wedge dz}{2i}\wedge \d\theta;
\end{equation}
note in particular that
\begin{equation}
	\d\Sigma^3 (X,-X_\perp, V) = 1.
	\label{eq:ONframe}
\end{equation}

The horocyclic parameterization \eqref{eq:hyp_fan_beam} induces a smooth bijection
\begin{equation}\label{eq:geod_param}
    \Gh \times \Rm\ni (\beta,a,t)\overset{F}{\longmapsto} (\gamma_{\beta,a} (t), \dot\gamma_{\beta,a}(t))\in S\Dm_H^\circ;
\end{equation}
it will be shown below that it is actually a diffeomorphism, and we will compute the pullback of the Liouville measure \eqref{eq:Liouville}.

\begin{lemma}\label{lm:pushforward}
    With $F$ defined in \eqref{eq:geod_param}, we have $F^*\d\Sigma^3= \d\beta\wedge \d a\wedge\d t$.
\end{lemma}
\begin{proof}
    We factor $F = \Upsilon\circ\Psi$ and compute $F^*\d\Sigma^3 = \Psi^*\Upsilon^*\d\Sigma^3$, where $\Upsilon$ is the vertex parameterization \eqref{eq:Hgeo} of geodesics:
    \begin{equation}
	\Upsilon\colon \add{(\Rm/2\pi \Zm)_{\omega}} \times (-1,1)_s\times \Rm_{t'} \to S\Dm_H^\circ, \qquad \Upsilon(\omega, s, t') \coloneqq (\gammav_{\omega,s}(t'), \dot{\gamma}^{\mathrm{v}}_{\omega,s}(t')),
    \end{equation}
    and $\Psi\colon \Gh \times \Rm_t \to \add{(\Rm/2\pi \Zm)_{\omega}} \times (-1,1)_s\times \Rm_{t'}$ is the map encoding the horocyclic-to-vertex reparameterization in Lemma \ref{lem:horo2vertex}, namely
    \begin{equation}\label{eq:Psi_bat}
	\Psi(\beta,a,t) = \left( \beta + \frac{\pi}{2}+\tan^{-1}a, - \tan \left( \frac{1}{2} \tan^{-1}a \right), t + \log \sqrt{1+a^2} \right).
    \end{equation}	
    To compute $\Upsilon^*\d\Sigma^3$, we have $\Upsilon_* \partial_{t'} = X_{\Upsilon(\omega,s,t')}$. Moreover, as $\Upsilon_* \partial_s$ and $\Upsilon_* \partial_\omega$ are variation fields of variations through geodesics, they take the form, see \cite[p. 91]{Paternain2023},
    \begin{eqnarray*}
	\Upsilon_* \partial_\omega &=& a_1(\omega,s,t') X + a_2 (\omega,s,t') X_\perp + a_3 (\omega,s,t') V \\
	\Upsilon_* \partial_s &=& b_1(\omega,s,t') X + b_2 (\omega,s,t') X_\perp + b_3 (\omega,s,t') V,
    \end{eqnarray*}
    where, using that the Gauss curvature satisfies $\kappa = -1$,
    \begin{equation}
	\dot a_1 = 0, \quad \dot a_2 + a_3 = 0, \quad \dot a_3 + a_2 = 0 \qquad (\dot{a}_j\coloneqq \d a_j/ \d t', \quad j=1,2,3),
    \end{equation}
    and similarly for $b_1, b_2, and b_3$. Then,
    \begin{eqnarray*}
	\Upsilon^* \d\Sigma^3 &=& \Upsilon^* \d\Sigma^3 (\partial_{t'}, \partial_\omega, \partial_s)\ \d t'\wedge \d\omega \wedge \d s \\
	&=& \d\Sigma^3 (\Upsilon_* \partial_{t'}, \Upsilon_* \partial_\omega, \Upsilon_* \partial_s)\ \d t'\wedge \d\omega \wedge \d s\\
	&\stackrel{\eqref{eq:ONframe}}{=}&  (b_2 a_3 - a_2 b_3)\ \d t'\wedge \d\omega \wedge \d s.
    \end{eqnarray*}
    Then, $b_2 a_3 - a_2 b_3$ is flow-invariant by Wronskian constancy, and equal to $(b_2 a_3 - a_2 b_3)(\omega,s,0)$. We then compute that $\partial_s|_{(s e^{i\omega}, \omega + \pi/2)} = \frac{2}{1-s^2} X_\perp$, hence $b_2(\omega,s,0) = \frac{2}{1-s^2}$ and $b_3(\omega,s,0) = 0$, and $\partial_\omega|_{(s e^{i\omega}, \omega + \pi/2)} = \frac{2s}{1-s^2} X + \frac{1+s^2}{1-s^2} V$, hence $a_2(\omega,s,0) = 0$ and $a_3(\omega,s,0) = \frac{1+s^2}{1-s^2}$. Thus, we arrive at
    \begin{equation}
	\Upsilon^* \d\Sigma^3 = 2 \frac{1+s^2}{(1-s^2)^2} \d t'\wedge \d s\wedge \d\omega.
    \end{equation}
    Now changing variables from vertex to horocyclic, where $2\tan^{-1}s = - \tan^{-1}a$ (hence $a = -\frac{2s}{1-s^2}$)
    and $\omega = \beta + \pi/2 + \tan^{-1}a$, one finds that
    \begin{equation}
	\d a = - (1+a^2) \frac{2\d s}{1+s^2} = - \frac{1+s^2}{(1-s^2)^2} 2\d s.
    \end{equation}
    Along with the relations $\d t' = \d t + \frac{a}{1+a^2}\d a$ and $\d\omega = \d\beta + \frac{1}{1+a^2}\d a$, we arrive at
    \begin{equation}
	2 \frac{1+s^2}{(1-s^2)^2} \d t'\wedge \d s\wedge \d\omega = -	\left(- \frac{1+s^2}{(1-s^2)^2} 2\d s\right) \wedge \d\omega \wedge \d t' = -\d a \wedge \d\beta\wedge \d t = \d\beta\wedge\d a\wedge\d t,
    \end{equation}
    hence the result.
\end{proof}

We immediately obtain a Santal\'o type formula.
\begin{proposition} \label{prop:SantaloH}
    Let $\Gh$ be the space of geodesics in the horocyclic parameterization, equipped with the measure $\d\beta\d a$ induced by the volume form $\d \Sigma_\partial \coloneqq \d\beta \wedge\d a$. Let $f\in C_c^\infty(S\Dm_H^{\circ} )$. Then,
    \begin{equation}
	\int_{\Gh} \int_\Rm f(\gamma_{\beta,a}(t), \dot \gamma_{\beta,a}(t))\ \d t\ \d \Sigma_\partial=\int _{S\Dm_H^\circ}f\, \d\Sigma^3.
	\label{eq:SantaloH}
    \end{equation}
\end{proposition}

\begin{proof}
    We have, by Lemma \ref{lm:pushforward} and using the change of variable $F$ from \eqref{eq:geod_param},
    \begin{equation}
	\int_{S\Dm_H^\circ}f\, d\Sigma^3 = \int_{\Gh} \int_\Rm f(\gamma_{\beta,a}(t),\dot\gamma_{\beta,a}(t))\ \d t\ \d \beta\ \d a;
    \end{equation}
    all integrals make sense because $f$ has compact support, which by virtue of \eqref{eq:xcontrol} implies that $(\beta,a)\mapsto \int_{\Rm} f(\gamma_{\beta,a}(t),\dot\gamma_{\beta,a}(t))\ \d t$ has compact support in $\Gh$.
\end{proof}

\subsection{Projective equivalence and intertwiners} \label{sec:ProjEq}

Recall the parameterization  of Euclidean geodesics through $\Dm_E$ given by \eqref{eq:Egeo}. It is well-known that the Euclidean disk and the Poincar\'e disk are projectively equivalent, explicitly so via the map $\Phi$ defined in \eqref{eq:Phi}.

\begin{proposition}[Explicit projective equivalence]\label{prop:projEquiv}
    With respect to vertex coordinates $\add{(\Rm/2\pi\Zm)}_\omega \times (-1,1)_s$, we have
    \begin{equation}
	\Phi (\gammav_{\omega,s}(t)) = \gamma^E_{\beta,\alpha}(u(t)), \qquad t\in \Rm,
	\label{eq:projEq}
    \end{equation}
    where
    \begin{equation}
	u(t) = \underbrace{\frac{1-s^2}{1+s^2}}_{\cos\alpha} \tanh t, \qquad \sin \alpha = \frac{-2s}{1+s^2}, \qquad \omega + \frac{\pi}{2} = \beta+\alpha+\pi,
	\label{eq:projEq_rel}
    \end{equation}
    with change of parameterization
    \begin{equation}
	\frac{\d u}{\d t} = \frac{1+s^2}{1-s^2}\ x^2(\gammav_{\omega,s}(t)), \qquad (x \text{ defined in \eqref{eq:xbdf}}).\index{$x$}
	\label{eq:reparam}
    \end{equation}
    In horocyclic coordinates ${\add{(\Rm/2\pi\Sm)_{\beta}}} \times \Rm_a$, with $t_0 = -\log \sqrt{1+a^2} = \log\muh$, we have
    \begin{equation}
	\Phi(\gamma_{\beta,a} (t+t_0)) = \gamma_{\beta,\tan^{-1} a}^E (u(t)), \qquad u(t) = \muh(a) \tanh t.
	\label{eq:projEq2}
    \end{equation}
\end{proposition}

\begin{proof}
    To prove \eqref{eq:projEq}-\eqref{eq:projEq_rel}, we aim at showing that $\Phi(\gammav_{\omega,s})$ is, up to reparameterization, a Euclidean geodesic. Note the preliminary identity:
    \begin{equation}
	1+|\gammav_{\omega,s}(t)|^2 = 1+ \frac{s^2 + \x^2}{1+s^2 \x^2} = \frac{(1+s^2)(1+\x^2)}{1+s^2 \x^2}, \qquad \x = \x(t).
    \end{equation}
    Using this, we directly compute
    \begin{equation}
	\begin{aligned}
	\Phi (\gammav_{\omega,s}(t)) &= 2 e^{i\omega} \frac{s+i\x}{1+is \x} \frac{1+s^2 \x^2}{(1+s^2)(1+\x^2)} = 2 e^{i\omega} \frac{(s+i\x)(1-is \x)}{(1+s^2)(1+\x^2)} \\
	&= e^{i\omega} \left( \frac{2s}{1+s^2} + i \frac{2\x}{1+\x^2} \frac{1-s^2}{1+s^2} \right) = e^{i(\omega + \pi/2)} \left( \frac{1-s^2}{1+s^2} \tanh t - i \frac{2s}{1+s^2} \right),
	\end{aligned}
    \end{equation}
    where we have used the fact that $\frac{2\x}{1+\x^2} = \frac{2\tanh(t/2)}{1+\tanh^2 (t/2)} = \tanh t$. Hence, \eqref{eq:projEq}-\eqref{eq:projEq_rel} hold.

    To prove \eqref{eq:reparam}, the reparameterization $u(t)$ from \eqref{eq:projEq_rel} satisfies
    \begin{equation}
	\frac{\d u}{\d t} = \frac{1-s^2}{1+s^2} \frac{1}{\cosh^2 t} \stackrel{\eqref{eq:xvertex}}{=} \frac{1+s^2}{1-s^2} x^2(\gammav_{\omega,s}(t)).
    \end{equation}
    Finally, \eqref{eq:projEq2} follows by combining \eqref{eq:projEq} and \eqref{eq:horo2vertex}. 	
\end{proof}

\begin{remark}
    As already mentioned, the curves given in \eqref{eq:projEq} are unit speed geodesics in the Beltrami-Klein model $(\Dm_E^\circ,g_K)$ of hyperbolic space, where
    \begin{equation}\label{eq:klein_cartesian}
        g_K=\frac{|dz|^2}{1-|z|^2}+\frac{(\Re(\bar{z} dz))^2}{(1-|z|^2)^2},
    \end{equation}
    as follows by the fact that $\Phi:(\Dm_H^\circ,g_H)\to (\Dm_E^\circ,g_K)$ is an isometry.
\end{remark}

\subsubsection{Intertwining of X-ray transforms}

\begin{theorem}\label{thm:ProjEq}
    Fix $\gamma>-1$. Then, we have the intertwining relation
    \begin{equation}
	I_0^H x^{2+2\gamma} \circ \Phi^* = \muh \Psihf^* \circ I_0^E d^\gamma \quad \text{on} \quad C^\infty(\Dm_E).
	\label{eq:XrayRelation2}
    \end{equation}
\end{theorem}

\begin{proof} Fix $(\beta,a)\in \Gh$ and first note that with $t_0 = \log \muh(a)$,
    \begin{equation}
	x(\gamma_{\beta,a}(t)) \stackrel{\eqref{eq:horo2vertex}}{=} x (\gammav_{\beta+\pi/2+\tan^{-1}a ,-a/(\sqrt{1+a^2}+1)} (t-t_0)) \stackrel{\eqref{eq:xvertex}}{=} \frac{\muh(a)}{\cosh (t-t_0)}.
	\label{eq:xhoro}
    \end{equation}
    Now, fix $f\in C^\infty(\Dm_E)$. By direct calculation, we see that
    \begin{eqnarray*}
	I_0^H x^{2+2\gamma} \Phi^* f (\beta,a) &\stackrel{\eqref{eq:bdfs}}{=}& \int_{\Rm} (d^{\gamma} f)(\Phi(\gamma_{\beta,a}(t)))\ x^2 (\gamma_{\beta,a}(t))\ \d t \\
	&\stackrel{\eqref{eq:projEq2}}{=}& \int_{\Rm} (d^\gamma f) (\gamma^E_{\beta,\tan^{-1}a}(u(t-t_0)))x^2(\gamma_{\beta,a}(t))\ \d t \\
	&\stackrel{\eqref{eq:xhoro}}{=}& \muh(a) \int_{\Rm} (d^\gamma f) (\gamma^E_{\beta,\tan^{-1}a}(u(t-t_0))) \frac{\muh(a)}{\cosh^2(t-t_0)}\ \d t,
    \end{eqnarray*}
    where $u(t) = \muh(a) \tanh(t)$. Upon changing variable $t\to u(t-t_0)$ with Jacobian $\frac{\muh(a)}{\cosh^2(t-t_0)}$, the last right-hand side equals $\muh(a) I_0^E d^\gamma f (\Psihf(\beta,a))$. Hence, \break the result.
\end{proof}

\subsubsection{Intertwining of backprojection operators}

\begin{theorem}\label{thm:interadj} We have the intertwining relation
    \begin{equation}
	    (I_0^H)^\sharp = \Phi^* \circ d^{1/2} (I_0^E)^\sharp \mu^{-2} \circ \Psihf^{-*} \quad \text{on }\quad C^\infty(\Gh),
	\label{eq:interadj}
    \end{equation}
    \add{where we use the shorthand notation $\Psihf^{-*}\coloneqq(\Psihf^{-1})^{*}$.}
\end{theorem}

\begin{proof} The proof combines Theorem \ref{thm:ProjEq} and Santal\'o's formula \eqref{eq:SantaloH}. On one hand, applying \eqref{eq:SantaloH} to the case where $f = h\circ \pi_H \cdot g\circ \pih$, with $h\in C_c^\infty (\Dm_H^{\circ})$ and $g\in C^\infty(\Gh)$, and where $\pi_H\colon S\Dm_H^\circ\to \Dm_H^\circ$ denotes the canonical projection, we get
    \begin{equation}
	\begin{split}
	    \int_{\Gh} (I_0^H h)\ g\ \d\Sigma_\partial &= \int_{S\Dm_H^\circ} h\circ\pi_H\ g\circ \pih\ \d\Sigma^3 \\
	    &= \int_{\Dm_H^\circ} h \int_{S_z\Dm_H} g\circ \pih(z,w)\d S_z(w)\ \d V_H = \int_{\Dm_H^\circ} h ((I_0^H)^\sharp g )\ \d V_H.
	\end{split}		
	\label{eq:firstComp}
    \end{equation}
    Further, from the definition of $\Phi$, we have
    \begin{equation}\label{eq:vol_change}
	\Phi^* \d V_E = \Phi^* (\rho\ \d\rho\ \d\omega) = \frac{2r}{1+r^2}2 \frac{1-r^2}{(1+r^2)^2}\ \d r\ \d\omega = x^3 c^{-2}\ r\ \d r\ \d\omega = x^3 \d V_H.
    \end{equation}
    From Theorem \ref{thm:ProjEq}, we also have
    \begin{equation}
	I_0^H h (\beta,\tan\alpha) = \cos\alpha\ I_0^E \tilde h (\beta,\alpha), \qquad \tilde h \coloneqq (h/x^2) \circ \Phi^{-1}.
    \end{equation}
    Thus, we can also compute, pulling back by $\Psihf^{-1}$,
    \begin{align*}
	\int_{\Gh} (I_0^H h) g\ \d \Sigma_\partial &= \int_{\partial_+ S\Dm_E} I_0^H h (\beta,\tan\alpha)\ g(\beta,\tan\alpha) \ \d\beta \frac{\d\alpha}{\cos^2 \alpha} \\
	&= \int_{\partial_+ S\Dm_E} \cos\alpha\ I_0^E \tilde h (\beta,\alpha)\ g(\beta,\tan\alpha) \ \d\beta \frac{\d\alpha}{\cos^2 \alpha} \\
	&\stackrel{(*)}{=} \int_{\Dm_E} \tilde{h} (I_0^E)^\sharp \left[ \frac{g(\beta,\tan\alpha) }{\cos^2\alpha} \right] \ \d V_E \\
	&= \int_{\Phi(\Dm_H)} (h/x^2) \circ \Phi^{-1} \cdot (I_0^E)^\sharp \left[ \frac{g(\beta,\tan\alpha) }{\cos^2\alpha} \right] \ \d V_E \\
	&= \int_{\Dm_H} \frac{h}{x^2} (I_0^E)^\sharp \left[ \frac{g(\beta,\tan\alpha) }{\cos^2\alpha} \right] \circ \Phi\ x^3 \d V_H \\
	&= \int_{\Dm_H} h \left(d^{1/2} (I_0^E)^\sharp \left[ \frac{g(\beta,\tan\alpha) }{\cos^2\alpha} \right]\right) \circ \Phi\ \d V_H.
    \end{align*}
    In the above, equality $(*)$ follows from $(I_0^E )^*=(I_0^E)^\sharp \mu^{-1}$ in the functional setting \eqref{eq:I0dgamma} with $\gamma=0$. The result follows by comparing the last right-hand side with that of \eqref{eq:firstComp}, noting that the equality of the two integrals holds true for every $h\in C_c^\infty (\Dm_H)$.
\end{proof}

%%%%%%%%%%%%%%%%%%%%%%%%%%%%%%%%%%%%%%%%%%%%%%%%%%%%%%%%%%%%%%%%%%%%%%%%%%%%%%%%%%%%%%%%%%%%%%%%%%%%%%%%%%%%

\section{Proofs of results in sections \ref{sec:extendedI0}-\ref{sec:Sobolev_Mapping_Properties}} \label{sec:normalOps}

\begin{proof}[Proof of Lemma \ref{prop:charac}] Note that $\Psihf\colon \Ghbar\to \partial_+ S\Dm_E$ is a diffeomorphism and $(\beta+ \tan^{-1}a, \muh) = (\beta+\alpha, \cos\alpha)\circ \Psihf$ with $(\beta+\alpha,\cos\alpha)$ globally defined and smooth on $\partial S\Dm_E$, and $\Psihf\circ\S^H_A= \S^E_A\circ \Psihf$. Therefore, the characterization in the statement is equivalent to showing that for $h\in C^\infty(\partial_+ S\Dm_E)$ satisfying $h\circ \S_A^E = h$, $h\in \Calp(\partial_+ S\Dm_E)$ if and only if its Taylor expansion off of $\partial_0 S\Dm_E$ relative to $\cos\alpha$ is only made of even terms. To prove the latter, consider a general function $\tilde{h}\in C^\infty(\partial S\Dm_E)$ expressed as a Fourier series (with rapid decay) in the variables $(\beta+\alpha,\alpha-\pi/2)$, i.e.,
    \begin{equation}
	\tilde{h} = \sum_{k\in \Zm} e^{ik(\beta+\alpha)} \bigg( a_{k,0} + \sum_{n\ge 1} \left(a_{k,n}\cos\left(n \left(\alpha-\frac{\pi}{2}\right)\right) + b_{k,n} \sin\left(n\left(\alpha-\frac{\pi}{2}\right)\right)\right) \bigg).
    \end{equation}
    Then, the condition $\tilde{h}\circ \S^E = \tilde{h}$ can be shown to annihilate all $b_{k,n}$'s. The resulting cosine series can be written as
    \begin{equation}
	\tilde{h} = \sum_{k\in \Zm, n\ge 0} a_{k,n} e^{ik(\beta+\alpha)} T_n (\cos(\alpha-\pi/2)) = \sum_{k\in \Zm, n\ge 0} a_{k,n} e^{ik(\beta+\alpha)} T_n (\sin\alpha),
    \end{equation}
    where $T_n$ is the $n$-th Chebychev polynomial of the first kind. In particular, near $\alpha = \pm \pi/2$, where $\sin \alpha = \pm \sqrt{1-\mu^2}$, $\tilde{h}$ is a smooth function of $(\beta+\alpha, \mu^2)$, whose Taylor expansion in $\mu$ only consists of even terms. We conclude that if $h$ is as before, $h\in \Calp(\partial_+ S\Dm_E)$ if and only if $\tilde{h}\coloneqq A_+h$ is smooth, that is, if and only if the Taylor expansion of $A_+h$ off of $\partial _+S\Dm_E$ (and thus of $h=A_+h\big|_{\partial_+S\Dm_E})$ contains only even terms.
\end{proof}

\begin{proof}[Proof of Theorem \ref{thm:smoothmapping}]
    The result follows from bearing in mind the mapping property \eqref{eq:mappingI0E}, and the factorization  \eqref{eq:XrayRelation2} of $I_0^H x^{2+2\gamma}$ by continuous maps
    \begin{equation}
	\Cev (\Dm_H) \stackrel{\Phi^{-*}}{\longrightarrow} C^\infty(\Dm_E) \stackrel{I_0^E d^\gamma}{\longrightarrow} \mu^{2\gamma+1} \Calp (\partial_+ S\Dm_E) \stackrel{\Psihf^{*}\mu}{\longrightarrow} \muh^{2\gamma+2} \Calp (\Ghbar).
	\label{eq:factor_smooth}
    \end{equation}
\end{proof}

\begin{proof}[Proof of Theorem \ref{thm:bounded}]
    For $\gamma>-1$, by \eqref{eq:XrayRelation2}, $I^H_0 x^{2+2\gamma}$ can be factored as
    \begin{equation}
	\begin{split}
	    L^2(\Dm_H, x^{2\gamma+3}\d V_H) &\stackrel{\Phi^{-*}}{\longrightarrow} L^2(\Dm_E, d^\gamma \d V_E) \\
	    & \stackrel{I_0^E d^\gamma}{\longrightarrow} L^2_+ (\partial_+ S\Dm_E, \mu^{-2\gamma}\d\beta\d\alpha) \stackrel{\Psihf^{*}\mu}{\longrightarrow} L^2_+ (\Gh, \muh^{-2\gamma}\d\beta\d a).
	\end{split}
	\label{eq:factor}
    \end{equation}
    Here, boundedness follows because the outermost maps are isometric isomorphisms, and by \eqref{eq:I0dgamma}, the middle map is bounded. Considering Hilbert adjoints, $(I_0^H x^{2+2\gamma})^*$ is obtained from $(I_0^E d^\gamma)^*$ by the adjoint diagram
    \begin{equation}
	\begin{split}
	    L^2_+ (\Gh, \muh^{-2\gamma}\d\beta\d a) &\stackrel{\mu^{-1} \Psihf^{-*}}{\longrightarrow} L^2_+ (\partial_+ S\Dm_E, \mu^{-2\gamma}\d\beta\d\alpha) \\
	    &\stackrel{(I_0^E d^\gamma)^*}{\longrightarrow} L^2(\Dm_E, d^\gamma \d V_E)  \stackrel{\Phi^{*}}{\longrightarrow}  L^2(\Dm_H, x^{2\gamma+3}\d V_H),	
	\end{split}	
	\label{eq:factor_adj}
    \end{equation}
    hence
    \begin{equation}
	\begin{split}
	(I_0^H x^{2+2\gamma})^* &= \Phi^* \circ (I_0^E d^\gamma)^* \circ \mu^{-1} \Psihf^{-*} \\
	&= \Phi^* \circ (I_0^E)^\sharp \mu^{-2\gamma-1} \circ \mu^{-1} \Psihf^{-*}\\
	&= x^{-1} \Phi^* \circ d^{1/2} (I_0^E)^\sharp \mu^{-2} \circ \Psihf^{-*} \muh^{-2\gamma} \stackrel{\eqref{eq:interadj}}{=} x^{-1} (I_0^H)^\sharp \muh^{-2\gamma}.
	\end{split}
    \end{equation}
\end{proof}

\begin{proof}[Proof of Theorem \ref{thm:SVDH}]
    For any $n\ge 0$ and $k\in \Zm$, we compute
    \begin{equation}
	\begin{split}
	(I_0^H x^{2\gamma+2})^* \psi_{n,k}^{\gamma,H}& = \Phi^* (I_0^E d^\gamma)^* \Psihf^{-*} \muh^{-1} \muh \Psihf^{*}\psi_{n,k}^\gamma \\ &= \Phi^* (I_0^E d^\gamma)^* \psi_{n,k}^\gamma
\overset{\text{Thm \ref{thm:SVD}}}{=} \left\{
	\begin{array}{ll}
	    0, & \text{if } k\notin \{0, \dots, n\}, \\
	    \Phi^* Z_{n,k}^\gamma & \text{otherwise}.
	\end{array}
	\right.
	\end{split}
    \end{equation}
    Since the maps
    \begin{align*}
           &\Phi^*\colon L^2(\Dm_E,d^\gamma \d V_E)\to L^2(\Dm_H,x^{2\gamma+3}\d V_H),\\
          &\muh\Psihf^{*}\colon  L^2(\partial_+S\Dm_E,\mu^{-2\gamma}\d\beta\d\alpha)\to L^2(\overline{\Gh},\muh^{-2\gamma}\d\beta\d a)
    \end{align*} are isometries, the family $\{\Phi^* Z_{n,k}^\gamma\}_{n,k}$ is orthogonal and $\{\muh\Psihf^{*} \psi_{n,k}^\gamma\}_{n,k}$ is orthonormal. Moreover, for $n\ge 0$ and $0\le k\le n$,
    \begin{equation}
	I_0^H x^{2\gamma+2} \widehat{\Phi^* Z_{n,k}^\gamma} = I_0^H x^{2\gamma+2} \Phi^* \widehat{Z_{n,k}^\gamma} \stackrel{\eqref{eq:XrayRelation2}}{=} \Psihf^{*} \mu I_0^E d^\gamma \widehat{Z_{n,k}^\gamma} = \Psihf^{*} \mu\ \sigma_{n,k}^\gamma \psi_{n,k}^\gamma = \sigma_{n,k}^\gamma \psi_{n,k}^{\gamma,H}
    \end{equation}
    by Theorem \ref{thm:SVD}, hence the result.
\end{proof}

\begin{proof}[Proof of Proposition \ref{prop:LgH}]
    Notice that $\L_\gamma^H = \Phi^* \circ \L_\gamma \circ \Phi^{-*}$ and that $\Phi^* \colon C^\infty(\Dm_E)\to C^{\infty}_{\mathrm{ev}}(\Dm_H)$ is an isomorphism, so naturally $\L_\gamma^H (\Cev (\Dm_H))\subset \Cev(\Dm_H)$. From the identity
    \begin{equation}
	(\L_\gamma^H f, g)_{L^2(\Dm_H, x^{2\gamma+3}\d V_H)} = (\L_\gamma \Phi^{-*}f,\Phi^{-*}g)_{L^2 (\Dm_E, d^\gamma\d V_E)}, \qquad f, g\in \Cev(\Dm_H),
    \end{equation}
    the $L^2(\Dm_H, x^{2\gamma+3}\d V_H)$-symmetry follows. Essential self-adjointness and spectral decomposition also follow straightforwardly from the properties of $\L_\gamma$, notably \cite[Theorem 6]{Mishra2022}.
\end{proof}

\begin{proof}[Proof of Proposition \ref{prop:TgH}] Notice that $\T_\gamma^H = \muh \circ \Psihf^{*} \circ \T_\gamma \circ \Psihf^{-*} \circ \muh^{-1}$, and that $\Psihf^{-*} \colon \Calp (\Ghbar) \to \Calp (\partial_+ S\Dm_E)$ is an isomorphism. By \cite[Eq.\ (28)]{Mishra2022},
    \begin{equation}
	\T_\gamma (\mu^{2\gamma+1} \Calp (\partial_+ S\Dm_E)) \subset \mu^{2\gamma+1} \Calp (\partial_+ S\Dm_E),
    \end{equation}
    and hence $\T_\gamma^H (\muh^{2\gamma+2} \Calp (\Ghbar)) \subset \muh^{2\gamma+2} \Calp (\Ghbar)$. From the identity
    \begin{equation}
	(\T_\gamma^H f, g)_{L^2_+(\G,\muh^{-2\gamma}\d\beta\d a)} = (\T_\gamma (\Psihf^{-*}(\muh^{-1}f)), (\Psihf^{-*}(\muh^{-1}f)))_{L^2(\partial_+ S\Dm_E, \mu^{-2\gamma}\d\beta\d\alpha)},
    \end{equation}
    valid for any $f,g\in \muh^{2\gamma+2} \Calp(\Ghbar)$, the $L^2_+(\G,\mu^{-2\gamma}\d\beta\d a)$-symmetry of $\T_\gamma^H$ follows from the $L^2(\partial_+ S\Dm_E, \mu^{-2\gamma}\d\beta\d\alpha)$-symmetry of $\T_\gamma$. Essential self-adjointness and spectral decomposition also follow from the properties of $\T_\gamma$, notably\break \cite[Theorem 7]{Mishra2022}.
\end{proof}

\begin{proof}[Proof of Proposition \ref{prop:interH}] Let $E_\gamma \coloneqq \muh^{2\gamma+2} \Calp(\Ghbar)$ and $F_\gamma \coloneqq \mu^{2\gamma+1} \Calp(\partial_+ S\Dm_E)$ for conciseness. Then, the statement of \eqref{eq:interadjH} is the outermost cycle in the following commutative diagram:
    \begin{center}
	\begin{tikzcd}
	    E_\gamma \arrow["{\T^H_\gamma}", bend left=24]{rrr} \arrow["{\Psihf^{-*} \muh^{-1}}"]{r} \arrow["{(I_0^H x^{2+2\gamma})^*}"]{d}  & F_\gamma \arrow["{\T_\gamma}"]{r} \arrow["{(I_0^E d^\gamma)^*}"]{d} & F_\gamma \arrow["{\muh \Psihf^*}"]{r} \arrow["{(I_0^E d^\gamma)^*}"]{d} & E_\gamma \arrow["{(I_0^H x^{2+2\gamma})^*}"]{d} \\
	    \Cev(\Dm_H) \arrow["{\Phi^{-*}}"]{r} \arrow["{\L_\gamma^H}", bend right=18]{rrr} & C^\infty(\Dm_E) \arrow["{\L_\gamma}"]{r} & C^\infty(\Dm_E)   \arrow["{\Phi^*}"]{r} & \Cev(\Dm_H), 		
	\end{tikzcd}
    \end{center}
%    \begin{center}
%	\begin{tikzcd}
%	    \muh^{2\gamma+2} \Calp(\Ghbar) \arrow["{\T^H_\gamma}", bend left=18]{rrr} \arrow["{\Psihf^{-*} \muh^{-1}}"]{r} \arrow["{(I_0^H x^{2+2\gamma})^*}"]{d}  & \mu^{2\gamma+1} \Calp(\partial_+ S\Dm_E) \arrow["{\T_\gamma}"]{r} \arrow["{(I_0^E d^\gamma)^*}"]{d} & \mu^{2\gamma+1} \Calp(\partial_+ S\Dm_E) \arrow["{\muh \Psihf^*}"]{r} \arrow["{(I_0^E d^\gamma)^*}"]{d} & \muh^{2\gamma+2} \Calp(\Ghbar) \arrow["{(I_0^H x^{2+2\gamma})^*}"]{d} \\
%	    \Cev(\Dm_H) \arrow["{\Phi^{-*}}"]{r} \arrow["{\L_\gamma^H}", bend right=18]{rrr} & C^\infty(\Dm_E) \arrow["{\L_\gamma}"]{r} & C^\infty(\Dm_E)   \arrow["{\Phi^*}"]{r} & \Cev(\Dm_H), 		
%	\end{tikzcd}
%    \end{center}
    where the left and right cycles follow from Theorem \ref{thm:interadj} and Theorem \ref{thm:bounded}, the middle cycle is \cite[Lemma 8]{Mishra2022}, and the top and bottom cycles are definitions. Then, \eqref{eq:interH} follows from \eqref{eq:interadjH} by considering adjoints, using the fact that $(\L_\gamma^H, \Cev(\Dm_H))$ is $L^2(\Dm_H, x^{2\gamma+3}\d V_H)$-symmetric and $(\T_\gamma^H, \muh^{2\gamma+2} C_{\alpha,+}^\infty(\Ghbar))$ is $L^2_+ (\Gh, \muh^{-2\gamma}\d\beta\d a)$-symmetric. 	
\end{proof}

\begin{proof}[Proof of Theorem \ref{thm:FuncRelH}]
    The result follows by combining equation \eqref{eq:signkgammaH} with the fact that, in the basis $\{\Phi^* Z_{n,k}^\gamma\}_{n\ge 0, 0\le k\le n}$, the following eigenequations hold (see \cite[Eq.\ (14)]{Mishra2022}):
    \begin{equation}
	\D_\gamma^H \Phi^* Z_{n,k}^\gamma = n \Phi^* Z_{n,k}^\gamma, \qquad D_\omega \Phi^* Z_{n,k}^\gamma = (n-2k) \Phi^* Z_{n,k}^\gamma, \quad n\ge 0,\ 0\le k\le n.
    \end{equation}
    Hence, both operators on either side of \eqref{eq:FuncRelH} agree on a complete orthogonal \break system.	
\end{proof}

\begin{proof}[Proof of Theorem \ref{thm:isomorphismH}]
    Combining \eqref{eq:factor} and \eqref{eq:factor_adj}, we arrive at the relation
    \begin{equation}
	(I_0^H x^{2+2\gamma})^* I_0^H x^{2+2\gamma} = \Phi^* \circ (I_0^E d^\gamma)^* I_0^E d^\gamma \circ \Phi^{-*}.
    \end{equation}
    Moreover, from the definitions \eqref{eq:Sobw}-\eqref{eq:Hsnorm}, one should notice that for any $s\ge 0$, the isomorphism $\Phi^{-*} \colon \Cev(\Dm_H) \to C^\infty(\Dm_E)$ extends to be an isometry $H_w^{s,\gamma} (\Dm_H) \stackrel{\approx}{\to} \wtH^{s,\gamma}(\Dm_E)$. Then, \ref{item:isomorphismH1} follows immediately from combining these facts with \eqref{eq:SobolevMapping}. To prove \ref{item:isomorphismH2}, we then write
    \begin{equation}
	\cap_{s\in \Rm} H_w^{s,\gamma} (\Dm_H) = \cap_{s\ge 0} \Phi^* \wtH^{s,\gamma} (\Dm_E) = \Phi^* \cap_{s\ge 0} \wtH^{s,\gamma} (\Dm_E)  \stackrel{(*)}{=}  \Phi^* C^\infty(\Dm_E) = \Cev(\Dm_H),
    \end{equation}
    where $(*)$ follows from \cite[Lemma 16.(c)]{Mishra2022}. Finally, \ref{item:isomorphismH3} is a direct consequence of \ref{item:isomorphismH1}-\ref{item:isomorphismH2}.
\end{proof}

\begin{proof}[Proof of Corollary \ref{cor:adj_onto}]
    Fix $\gamma=0$ (note that any value of $\gamma$ would give the same result). By Theorem \ref{thm:isomorphismH}\ref{item:isomorphismH3}, for $f\in x\Cev(\Dm_H)$, there exists $h\in \Cev(\Dm_H)$ such that
    \begin{equation}
	x^{-1} (I_0^H)^\sharp I_0^H x^2 h = (I_0^H x^2)^* I_0^H x^2 h = x^{-1} f \quad \implies \quad (I_0^H)^\sharp (I_0^H x^2 h) = f,
    \end{equation}
    where by Theorem \ref{thm:smoothmapping}, $I_0^H (x^2 h) \in \muh^2 \Calp(\Ghbar)$. The conclusion follows.

    Another proof makes use of Theorem \ref{thm:interadj} and the fact that the Euclidean backprojection operator $(I_0^{E})^\sharp \colon \Calp (\partial_+ S\Dm_E)\to C^\infty(\Dm_E)$ is onto, see e.g. \break\cite[Theorem 1.4]{Pestov2005}.
\end{proof}

\begin{proof}[Proof of Theorem \ref{thm:moments}] Since $u\in L^2_+ (\Gh, \muh^{-2\gamma}\d\beta\d a)$, we have an $L^2$-convergent expansion
    \begin{equation}
	u = \sum_{n\ge 0}\sum_{k\in \Zm} u_{n,k} \psi_{n,k}^{\gamma,H}, \qquad \sum_{n\ge 0} \sum_{k\in \Zm} |u_{n,k}|^2<\infty.
	\label{eq:udecomp}
    \end{equation}
    For $n\in \Nm_0$, let us denote
    \begin{equation}
	M_n(\omega) \coloneqq \int_{-1}^1 p_n^\gamma \left( \frac{-2s}{1+s^2} \right) u^{\rmv}(\omega,s) \frac{2\d s}{1+s^2}.
    \end{equation}
    The symmetry condition $u\circ \S_A^H = u$ is equivalent to $u^\rmv(\omega,s) = u^{\rmv}(\omega+\pi, -s)$, and this implies the symmetry $M_n(\omega+\pi) = (-1)^n M_n(\omega)$ for all $n \in \Nm_0$, using that $p_n^\gamma(-z)=(-1)^np_n^\gamma(z)$. In particular, the Fourier series of $M_n$ only has terms of same parity as $n$, and we write $M_n(\omega) = \sum_{k\in \Zm} M_{n,k} e^{i(n-2k)\omega}$, where
    \begin{equation}
	\begin{split}
	    M_{n,k} &= \frac{1}{2\pi} \int_{0}^{2\pi} e^{-i(n-2k)\omega} M_n(\omega)\d\omega \\
	&= \frac{1}{2\pi} \int_{0}^{2\pi} \int_{-1}^1 e^{-i(n-2k)\omega} p_n^\gamma \left( \frac{-2s}{1+s^2} \right) u^{\rmv}(\omega,s) \frac{2\d s}{1+s^2}\d \omega \\
	&= \frac{1}{2\pi} \int_{0}^{2\pi} \int_{-1}^1 u^{\rmv}(\omega,s) \overline{(\psi_{n,k}^{\gamma,H})^\rmv} (\omega,s) \left( \frac{1-s^2}{1+s^2} \right)^{-2\gamma-2} \frac{2\d s}{1+s^2}\d\omega \\
	&= \frac{1}{2\pi} \int_{\Gh} u\ \overline{\psi_{n,k}^{\gamma,H}} \muh^{-2\gamma}\d\beta\d a \\
	&= \frac{1}{2\pi} u_{n,k}.
	\end{split}
    \end{equation}
    Note that $M_n$ is \add{the restriction to $\Sm^1$ of} a homogeneous polynomial \add{on $\Rm^2$} of degree $n$ if and only if $M_{n,k} = 0$ for $k\notin \{0, \dots, n\}$, and this is equivalent to $u_{n,k} = 0$ for all $n$ and $k\notin \{0, \dots, n\}$. Therefore, $u$ satisfies the moment conditions if and only if it is orthogonal to the cokernel of $I_0^H x^{2+2\gamma}$.
    Therefore, $u= I_0^H x^{2\gamma+2} f$ for some $f\in H_w^{s,\gamma}(\Dm_H)$ (resp. $f\in \Cev(\Dm_H)$) if and only if it satisfies the moment conditions, and in addition has the required regularity, which is equivalent to \eqref{eq:Hscond} (resp. \eqref{eq:Cinfcond}) by Corollary \ref{prop:rangeSobolevH} and the equality $u_{n,k} = 2\pi M_{n,k}$.
\end{proof}

\section{Data space structure and boundary operators - case \texorpdfstring{$\gamma=0$}{gamma=0}}\label{sec:boundary_ops}

In this section, we study the structure of the data space for the X-ray transform on $(\Dm_H^\circ,g_H)$ and construct boundary operators analogous to those on the Euclidean disk. We conclude it with the proof of Theorem \ref{thm:gammazeroHalt}.

\subsection{Incoming and outgoing boundary for the hyperbolic sphere bundle; the scattering map.}\label{ssec:scattering_map}

Recall that the  space of geodesics on $\Dm_H^\circ$ is identified with $\Gh= \add{(\Rm/2\pi\Zm)}_\beta \times \Rm_a$ via the horocyclic parameterization \eqref{eq:hyp_fan_beam}. We set
\begin{equation}
    \Gamma_\pm \coloneqq \add{\Gh}\times\{\pm 1\}, \qquad \Gamma \coloneqq \Gamma_+\cup \Gamma_-.
\end{equation}
The spaces $\Gamma_{+/-}$ are analogues of the interior of the incoming/outgoing boundaries $(\partial_{+/-} SM)^{\circ}$ defined for a simple Riemannian surface $M$, and the map $(\beta,a)\mapsto (\beta,a,1)$ identifies $\Gh$ with the incoming part $\Gamma_+$.
Note that by Lemma \ref{lm:cosphere-horocycle} and the discussion preceding it, there is already an identification of $\Gh$ with  $\partial_+ S^*\Dm_H^\circ$ which induces a natural diffeomorphism  $\Gamma_\pm \to \partial_\pm S^*\Dm_H^\circ$, $(\beta,a,\pm 1)\mapsto (e^{i\beta}, \pm\frac{\d \tilde{x}}{ \tilde{x}}-a \d \omega)$.

Additionally, by virtue of Lemma \ref{lm:cosphere-horocycle}, the hyperbolic scattering relation, to be understood as the map $\S^H \colon \Gamma_\pm \to \Gamma_{\mp}$ such that $\S^H (g(\dot\gamma(\pm\infty))) = g(\dot\gamma(\mp\infty))$ for any unit-speed geodesic $\gamma(t)$ on $\Dm_H^\circ$ is expressed as
\begin{equation}
    \mathcal{S}^H :\Gamma_\pm \to \Gamma_\mp, \qquad (\beta, a,\pm 1)\mapsto \big(\beta+\pi \pm 2\tan^{-1} a, a, \mp 1\big).
    \label{eq:SH}
\end{equation}
Finally, the antipodal involution $A_H$ on the fibers of $\Gamma$ takes the form
\begin{equation}\label{eq:hyp_antip_invol}
    A_H:\Gamma_\pm\to \Gamma_{\mp}, \qquad (\beta,a,\pm 1)\mapsto (\beta,-a,\mp 1),
\end{equation}
so that $\S_A^H = \S^H \circ A_H = A_H \circ \S^H$.

In terms of the identification described above, the map $\Psihf:\Gh\to (\partial_+S\Dm_E)^\circ$ in \eqref{eq:Psihf} can be viewed as a map $\Gamma_+\to (\partial_+S\Dm_E)^\circ$.
We extend it to $\Gamma$ by setting $\Psihf(\zeta) \coloneqq A_E\circ \Psihf\circ A_H(\zeta)$ for $\zeta\in \Gamma_-$, with $A_E$ the Euclidean antipodal map defined in \eqref{eq:AE}. In this way, we obtain a diffeomorphism $ \Psihf:\Gamma\to \partial  S\Dm_E\setminus \partial_0 S\Dm_E$
\begin{equation}
	\Psihf(\beta,a,\lambda)=\begin{cases}
		(\beta	,\tan^{-1}(a)),\quad& (\beta,a,\lambda)\in \Gamma_+,\\
		(\beta	,\pi-\tan^{-1}(a)),\quad &(\beta,a,\lambda)\in \Gamma_-,
	\end{cases}
	\label{eq:PsihfGamma}
\end{equation}
see also Figure \ref{fig:Psihf}.
It is not hard to check the following intertwining properties:
\begin{equation}\label{eq:intertwining_HE}
    \Psihf\circ\mathcal{S}^H= \mathcal{S}^E\circ \Psihf,\qquad \Psihf\circ\mathcal{S}^H_A= \mathcal{S}^E_A\circ \Psihf.
\end{equation}
Below, $\muh$ is extended to $\Gamma$ so that it becomes odd with respect to $A_H$, i.e., $\muh(\beta,a,\pm 1)=\pm(1+a^2)^{-1/2}$. This also implies that $\muh\circ \S^H=-\muh$.

\begin{figure}[htpb]
    \begin{center}
	\includegraphics[width=0.95\textwidth]{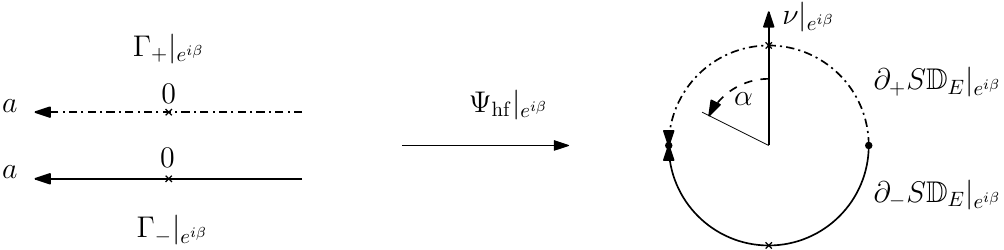}
    \end{center}
    \caption{The fiberwise picture for the map $\Psihf$ defined in \eqref{eq:PsihfGamma}.}
    \label{fig:Psihf}
\end{figure}

\subsection{The boundary operators \texorpdfstring{$A_\pm^H$, $C_\pm^H$, and ${P}_-^H$}{A pm, C m, and P m} } \label{ssec:boundaryOps}

Analogously to \eqref{eq:ApmE}, for a measurable function on $\Gamma_+$, we let
\begin{equation}
    A_\pm^H u(\beta,a,\lambda) \coloneqq \begin{cases}
        u(\beta,a,\lambda),\quad &(\beta,a,\lambda)\in \Gamma_+\\
        \pm u(\mathcal{S}^H (\beta,a,\lambda)),\quad &(\beta,a,\lambda)\in \Gamma_-,
    \end{cases}
\end{equation}
where $\S^H$ is defined in \eqref{eq:SH}. Then, $A_\pm^H u$ is a measurable function on $\Gamma$.
The $L^2(\Gamma_+,\muh^\sigma \d\beta\d a)\to L^2(\Gamma,|\muh|^\sigma \d\beta\d a)$ adjoint for any $\sigma \in \Rm$ is
\begin{equation}
    ({A}_\pm^H)^*u(\beta, a, 1)=u(\beta, a, 1)\pm (\mathcal{S}^H)^* u(\beta, a, 1).
\end{equation}
The fiberwise odd Hilbert transform on $\Gamma$ is defined for $u\in C_c^\infty (\Gamma)$ by
\begin{equation}
    \mathcal{H}_- u(\beta,a,\pm 1) \coloneqq \pm \frac{1}{\pi }\mathrm{p.v.}\int_{-\infty}^\infty \frac{u_-(\beta,a',\pm 1)}{a-a'}\d a', \qquad u_- \coloneqq (u- A_H^*u)/2.
\end{equation}
It extends to a bounded operator on $L^2(\Gamma,\d \beta\d a)$ with the property $A_H^*\mathcal{H}_-=-\mathcal{H}_-$. Using it, we define the following operators on {$\Gamma_+$}:
\begin{equation}\label{eq:CPH}
    {C}_-^H \coloneqq \frac{1}{2} (A_-^H)^*\H_{-}A_-^H,\qquad {P}_-^H \coloneqq  (A_-^H)^*\H_{-}A_+^H.
\end{equation}
In the statement below, we denote by $H_-$ the fiberwise Hilbert transform on the fibers of $\partial S\Dm_E$, restricted to fiberwise odd integrands.

\begin{lemma}\label{lm:Hilbert_intertwiner}
    For $u\in L^2(\Gamma,\d\beta\d a)$, one has
    \begin{equation}\label{eq:Hilbert_intertwiner}
	|\muh|\Psihf^{*}H_-\Psihf^{-*}|\muh|^{-1}u = \mathcal{H}_-u.
    \end{equation}
\end{lemma}

\begin{proof}
    First assume that $u\in C^\infty_c(\Gamma)$.
    Let $(\beta,a,1)\in \Gamma_+$ so that $\Psihf(\beta,a,1)=(\beta,\tan^{-1} (a))$.
    We use the expression (see the proof of \cite[Lemma 9.4.15]{Paternain2023})
    \begin{equation}
	H_-f(z,w)=\frac{1}{2\pi}\mathrm{p.v.}\int_{S_z\Dm_E}\frac{f(z,v)}{\langle v,-w^\perp \rangle_E}\d S_z(v),\quad f\in C^\infty(S\Dm_E), \quad (z,w)\in S\Dm_E,
	\label{eq:Hminus}
    \end{equation}
    where $\langle\cdot,\cdot\rangle_E$ is the Euclidean metric. We have
    \begin{equation}
	\begin{split}
	((\Psihf^{*}H_-\Psihf^{-*})|\muh|^{-1}u)(\beta,a,1)&\hspace{-.1 in}\overset{\phantom{v=e^{i\alpha'}}}{=}\hspace{-.13 in}
	\frac{1}{2\pi}\mathrm{p.v.}
	\int_{ S_{e^{i\beta}}\Dm_E }
	\frac{\Psihf^{-*}(|\muh|^{-1}u)(e^{i\beta},v)}{\langle v, -2\Re(i e^{i\tan^{-1} (a)}\partial_z)\rangle_E }\d S_{e^{i\beta}}(v)\\
	&\hspace{-.1 in}\overset{v=e^{i\alpha'}}{=}\frac{1}{2\pi}\mathrm{p.v.}
	\int_{-\pi/2}^{\pi/2}
	\frac{\Psihf^{-*}(|\muh|^{-1}u)(\beta,\alpha')}{\Re(i e^{i(\alpha'-\tan^{-1} (a))}) }\d \alpha'\\ &\qquad  +\frac{1}{2\pi}\mathrm{p.v.}
	\int_{\pi/2}^{3\pi/2}
	\frac{\Psihf^{-*}(|\muh|^{-1}u)(\beta,\alpha')}{\Re(i e^{i(\alpha'-\tan^{-1} (a))}) }\d \alpha'.
	\end{split}
    \end{equation}
    We now use the change of variables $\alpha'=\tan^{-1}(a') $ and $\alpha'=\tan^{-1}(a') +\pi $ in the first and second principal value integral, respectively, to find
    \begin{equation}
\begin{split}
	&((\Psihf^{*}H_-{\Psi}_h^{*})|\muh|^{-1}u)(\beta,a,1) \\
	&= \frac{1}{2\pi}\mathrm{p.v.} \int_{-\infty}^{\infty} \frac{\sqrt{1+a'^2}\, u(\beta,a',1)}{\Re(i e^{i(\tan^{-1}(a')-\tan^{-1} (a))}) }\frac{1}{1+a'^2}\d a'\\*
    &\qquad
     +\frac{1}{2\pi}\mathrm{p.v.}
    \int_{-\infty}^{\infty}
    \frac{\sqrt{1+a'^2}\, u(\beta,-a',-1)}{-\Re(i e^{i(\tan^{-1}(a')-\tan^{-1} (a))}) }\frac{1}{1+a'^2}\d a'\\
    &=\frac{1}{2\pi}\mathrm{p.v.}
    \int_{-\infty}^{\infty}
    \frac{u(\beta,a',1)}{\sin (\tan^{-1}({a})-\tan^{-1}{(a')}) }\frac{1}{\sqrt{1+a'^2}}\, \d a'\\*
    &\qquad +\frac{1}{2\pi}\mathrm{p.v.}
    \int_{-\infty}^{\infty}
    \frac{u(\beta,-a',-1)}{-\sin (\tan^{-1}({a})-\tan^{-1}{(a')})  }\frac{1}{\sqrt{1+a'^2}}\, \d a'\\
    &=\frac{1}{\pi}\mathrm{p.v.}
    \int_{-\infty}^{\infty}
    \frac{u_{-}(\beta,a',1)}{\sin (\tan^{-1}({a})-\tan^{-1}{(a')}) }\frac{1}{\sqrt{1+a'^2}}\, \d a'.
\end{split}
   \end{equation}
   We now have $\sin \big(\tan^{-1}({a})-\tan^{-1}{(a')}\big) =\frac{a-a'}{\sqrt{1+a^2}\sqrt{1+a'^2}}$, (using $\sin(\tan^{-1}(a))=\frac{a}{\sqrt{1+a^2}}$ and $\cos(\tan^{-1}(a))=\frac{1}{\sqrt{1+a^2}}$), hence
   \begin{equation}
       (|\muh|\Psihf^{*}H_-\Psihf^{-*}|\muh|^{-1}u)(\beta,a,1)=\frac{1}{\pi}\mathrm{p.v.}
       \int_{-\infty}^{\infty}
       \frac{u_{-}(\beta,a',1)}{a-a' } \d a'=\mathcal{H}_-u(\beta,a,1).
   \end{equation}

   Now, note that $A_H$ and the Euclidean antipodal map $A_E$ satisfy $ A_E\circ\Psihf=\Psihf\circ A_H$. Since $A_E^*H_-=-H_-$ and $ A_H^*|\muh|=|\muh| A_H^*$, we see that
   \begin{equation}
	\begin{split}
       (|\muh|\Psihf^{*}H_-\Psihf^{-*}|\muh|^{-1}u)(\beta,-a,-1)&=(A_H^*|\muh|\Psihf^{*}H_-\Psihf^{-*}|\muh|^{-1}u)(\beta,a,1)\\
       &=-\mathcal{H}_-u(\beta,a,1)=\mathcal{H}_-u(\beta,-a,-1),
	\end{split}
   \end{equation}
   and this proves the statement for $u\in C^\infty_c(\Gamma)$.

   For $u\in L^2(\Gamma,\d\beta\d a)$, the result follows by density. Indeed, it is enough to notice that $H_-$ and $\mathcal{H}_-$ are bounded on $L^2(\partial S \Dm_E,\d \beta\d \alpha)$ and $L^2(\Gamma,\d\beta\d a)$, respectively, and that $\Psihf^{-*} |\muh|^{-1}: L^2(\Gamma,\d\beta\d a)\to L^2(\partial S \Dm_E,\d \beta\d \alpha) $ is an isometry.
\end{proof}

We can now prove the following, which, together with the mapping properties for $C_-$, $P_-$ mentioned in Section\ \ref{sec:gammazero} and Lemma \ref{lm:HTDNH} below, imply the mapping properties in \eqref{eq:CPCinf} and \eqref{eq:CP}.

\begin{proposition}\label{prop:intertwining_bd_ops}
    For $u\in L^2(\Gamma_+,\d\beta\d a)$ and $v\in L^2(\Gamma,\d\beta\d a)$, one has
    \begin{align}
	\label{eq:AplusminusA}
	A_\pm^Hu=\Psihf^{*}A_\pm\Psihf^{-*}u,&\quad
	({A}_\pm^H)^* v= \Psihf^{*}A_\pm^*\Psihf^{-*}v,\\
	{C}_-^H u	=\muh\Psihf^{*}C_-\Psihf^{-*}\muh^{-1}u ,&\quad
	{P}_-^H u= \muh\Psihf^{*}P_-\Psihf^{-*}\muh^{-1}u.\label{eq:CPsminusB}
    \end{align}
\end{proposition}

\begin{proof}
    Eq \eqref{eq:AplusminusA} is immediate by definition of $A_\pm$, $A_\pm^*$ (\ref{eq:ApmE}-\ref{eq:A_star}) and  \eqref{eq:intertwining_HE}.
    Using it, we compute
    \begin{equation}
	\begin{split}
	    \muh \Psihf^{*}A_-^* H_-A_\pm \Psihf^{-*} \muh^{-1} &
	    =\muh(\Psihf^{*}A^*_-\Psihf^{-*})(\Psihf^{*}H_-\Psihf^{-*})(\Psihf^{*}A_\pm\Psihf^{-*})\muh^{-1} \\
	    &= \muh({A}_-^H)^*(\Psihf^{*}H_-\Psihf^{-*})A_\pm^H\muh^{-1} \\
	    &=({A}_-^H)^*|\muh|(\Psihf^{*}H_-\Psihf^{-*})|\muh|^{-1}A_\pm^H \\
	    &=({A}_-^H)^*\H_-A_\pm^H,	
	\end{split}
	\label{eq:Cminus_eucl_comp}
    \end{equation}
    by Lemma \ref{lm:Hilbert_intertwiner}.
    Here, we used that $A_\pm^H \muh^{-1}=|\muh|^{-1}A_\pm^H$ and $(A_-^H)^*|\muh|=\muh (A_-^H)^*$.
    Now, \eqref{eq:CPsminusB} follows by \eqref{eq:Cminus}, \eqref{eq:Pminus}, and \eqref{eq:CPH}.
\end{proof}

By Proposition \ref{prop:intertwining_bd_ops} and Eqs. \eqref{eq:Cminus_spectral} and \eqref{eq:SVDP}, the following is immediate:.

\begin{corollary}
    The operators
	\begin{align*}
		{C}_-^H \colon L^2_+ (\Gh, \d\beta\d a)\to L^2_+ (\Gh, \d\beta\d a), \quad
		{P}_-^H \colon L^2_- (\Gh, \d\beta\d a)\to L^2_+ (\Gh, \d\beta\d a),
	\end{align*}
	have spectral decompositions of the following form:
    \begin{equation}
	{C}_-^H \psi_{n,k}^{H} = i (1_{k<0} - 1_{k>n}) \psi_{n,k}^{H}, \quad
	{P}_-^H \phi_{n,k}^H = -2i \ 1_{0\le k\le n}\ \psi_{n,k}^{H}, \quad n\ge 0,\ k\in \Zm,
	\label{eq:SVDCPH}
    \end{equation}
    where $\psi_{n,k}^H  \coloneqq \muh \Psihf^{*}\psi_{n,k}$, $\phi_{n,k}^H  \coloneqq  \muh \Psihf^{*}\phi_{n,k}$, with $\psi_{n,k}, \phi_{n,k}$ defined in \eqref{eq:phipsiE}.
\end{corollary}

\subsection{The compactification \texorpdfstring{$\overline{\Gamma}$ and an ambient characterization of $C_{\alpha,+}^\infty (\Ghbar)$}{Gamma bar}} \label{ssec:smoothstruct}

In this section, we construct a compactification of $\Gamma$ ($\overline{\Gamma}$, defined in \eqref{eq:Gammabar}), and upon equipping $\overline{\Gamma}$ with a smooth structure that makes it diffeomorphic to $\partial S\Dm_E$ (see Lemma \ref{lm:diffeo}), we show in Proposition \ref{prop:ambient} that $C_{\alpha,+}^\infty (\Ghbar)$ admits an ``ambient'' characterization akin to the statement
\begin{equation}
    C_{\alpha,+}^\infty(\partial_+ S\Dm_E) = \{u\in C^\infty(\partial_+ S\Dm_E)\colon A_+ u \in C^\infty(\partial S\Dm_E),\  \add{u\circ \S_A^E} = u\},
\end{equation}
which holds in the Euclidean case.

Consider the compactification $\overline{\Rm}_a=[-\infty,\infty]_a$ of the real line with the smooth structure described in Section \ref{sec:extendedI0} so that the map $\tan^{-1}:\Rm_a\to (-\pi/2,\pi/2)$ extends to a smooth diffeomorphism $\overline{\Rm}_a\to [-\pi/2,\pi/2]$, still denoted  by $\tan^{-1}$. Let $\overline{\Gamma}_{\pm} \coloneqq  \add{(\Rm/2\pi\Zm)}_\beta\times \overline{\Rm}_a\times  \{\pm 1\}$, and set
\begin{equation}
    \overline{\Gamma} \coloneqq \overline{\Gamma}_{+}\sqcup \overline{\Gamma}_-/_\sim
    \label{eq:Gammabar}
\end{equation}
where $\sim$ is the smallest equivalence relation on $\overline{\Gamma}_{-}\sqcup \overline{\Gamma}_{+}$ such that $(\beta,\pm \infty, 1)\sim (\beta,\pm \infty, -1)$. We equip $\overline{\Gamma}$ with the quotient topology, and we now construct a smooth structure on $\overline{\Gamma}$ making it diffeomorphic to $\partial S\Dm_E$. Consider the following subsets of $\overline{\Gamma}$:
\begin{equation}
	\begin{split}
	    U_1 & \coloneqq \big(\add{(\Rm/2\pi\Zm)}_\beta\times (-\infty,\infty]\times \{1\}\big)/_\sim\cup\big( \add{(\Rm/2\pi\Zm)}_\beta\times (-\infty,\infty]\times \{-1\}\big)/_\sim,\\
               U_2 & \coloneqq \big(\add{(\Rm/2\pi\Zm)}_\beta\times [-\infty,\infty)\times \{1\}\big)/_\sim\cup\big( \add{(\Rm/2\pi\Zm)}_\beta\times [-\infty,\infty)\times \{-1\}\big)/_\sim.
	\end{split}
\end{equation}
We obtain-well defined homeomorphisms $\varphi_j:U_j\to \varphi_j(U_j)\subset \add{(\Rm/2\pi\Zm)}_\beta\times \Rm$ by setting
\begin{equation}
\begin{aligned}
     \varphi_1(
    [\beta, a,\lambda])&=\begin{cases}
        (\beta,\tan^{-1}({a})), & \lambda=1\\
        (\beta,\pi -\tan^{-1}({a})), & \lambda=-1
    \end{cases},\\
    \varphi_2(
    [\beta, a,\lambda])&=\begin{cases}
        (\beta,\tan^{-1}({a})), & \lambda=1\\
        (\beta,-\pi
        -\tan^{-1}({a})), & \lambda=-1
    \end{cases}.
    \end{aligned}
\end{equation}
For $(\beta,\alpha)\in \varphi_2(U_1\cap U_2)=\add{(\Rm/2\pi\Zm)}_\beta\times \big((-\pi/2,\pi/2)\cup (-3\pi/2,-\pi/2)\big)$, they satisfy
\begin{equation}
    \varphi_1\circ \varphi^{-1}_2(\beta,\alpha)=
    \begin{cases}
	(\beta,\alpha),& (\beta,\alpha)\in \add{(\Rm/2\pi\Zm)}_\beta\times (-\pi/2,\pi/2),\\
	(\beta,\alpha+2\pi),& (\beta,\alpha)\in \add{(\Rm/2\pi\Zm)}_\beta\times (-3\pi/2,-\pi/2),
    \end{cases}
\end{equation}
which is a diffeomorphism onto $$\varphi_1(U_1\cap U_2)=\add{(\Rm/2\pi\Zm)}_\beta\times \big((-\pi/2,\pi/2)\cup (\pi/2,3\pi/2)\big).$$
Since $\overline{\Gamma}\subset U_1\cup U_2$, we obtain a smooth structure on $\overline{\Gamma}$ upon composing the first coordinate of $\varphi_j$ with charts for $\add{(\Rm/2\pi\Zm)_\beta}$ and extending the resulting charts for $\overline{\Gamma}$ to a maximal atlas.

\begin{lemma}\label{lm:diffeo}
    With the smooth structure for $\overline{\Gamma}$ defined above, $\overline{\Gamma}_\pm/_\sim$ become embedded smooth submanifolds with boundary,
    and the map $\Psihf^{-1}: \partial S\Dm_E\setminus \partial_0 S\Dm_E\to \Gamma$ extends to a smooth diffeomorphism
    \begin{equation}
        \bPsifh:\partial S\Dm_E\to \overline{\Gamma}, \quad (z,w)\mapsto \begin{cases}
            [\Psihf^{-1}(z,w)], \quad& (z,w)\in( \partial_\pm S\Dm_E)^{\circ}\\
            [(z,\pm \infty,1)],\quad& (z,w)=(z,\pm \nu^\perp) \in \partial_0 S\Dm_E,
        \end{cases}
    \end{equation}
    with inverse denoted by $\bPsihf$, where $\nu$ is the inward pointing unit normal vector field.
\end{lemma}

\begin{proof}
    For the  first statement, it suffices to construct charts of $\overline{\Gamma}$ that restrict to boundary charts for $\overline{\Gamma}_\pm$.
    For  $j=1,2$ and $|a|>0$, we check that
    \begin{equation}
	\begin{split}
	\varphi_j\big|_{\overline{\Gamma}_+}(\beta,a,1)&=(\beta,\tan^{-1}(a))=\big(\beta,\mathrm{sign}(a)(\pi/2-\arcsin(|\muh|)\big),\\
	\varphi_j\big|_{\overline{\Gamma}_-}(\beta,a,-1)&=(\beta,(-1)^{j+1}\pi-\tan^{-1}(a))\\
	&=\Big(\beta,(-1)^{j+1}\pi-\mathrm{sign}(a)(\pi/2-\arcsin(|\muh|)\big)\Big).
	\end{split}
    \end{equation}
    Since  by definition of the  smooth structure on $\overline{\Gamma}_\pm$, the function  $\arcsin(|\muh|)$ is   a smooth bdf for $\overline{\Gamma}_\pm$ for small $|\muh|$, we conclude that $\varphi_j\big|_{\overline{\Gamma}_\pm}$ yield smooth boundary charts upon suitably reflecting and translating the second variable.

    For the second statement, bijectivity is clear since $\Psihf^{-1}\big|_{(\partial_\pm S\Dm_E)^{\circ}}:(\partial_\pm S\Dm_E)^{\circ}\to \Gamma_\pm$ is a smooth diffeomorphism, and the equivalence relation $\sim$ is trivial on $\Gamma_\pm \subset \overline{\Gamma}_\pm$.
    We check the smoothness of $\varphi_1\circ \bPsifh$ near $\partial_0 S\Dm_E$.
    In fan-beam coordinates with $\alpha\in (0,\pi)$ so that $\bPsifh(\beta,\alpha)\in U_1$,
    \begin{equation}
	\begin{split}
	\varphi_1\circ \bPsifh(\beta,\alpha)
	&=\begin{cases}
	    \varphi_1( [\beta,\tan(\alpha),1)]), & \alpha \in (0,\pi/2) \\
	    \varphi_1(  [\beta,-\tan(\alpha),-1]), & \alpha \in (\pi/2,\pi ) \\
	    \varphi_1([(\beta,+ \infty,1)]),& \alpha=\pi/2
	\end{cases}\\
	&=\begin{cases}
	    (\beta,\alpha), & \alpha \in (0,\pi/2) \\
	    (\beta,\pi +\tan^{-1}(\tan(\alpha))), & \alpha \in (\pi/2,\pi ) \\
	    (\beta,\pi/2),& \alpha=\pi/2
	\end{cases}
	.
	\end{split}
    \end{equation}
    Since  $\tan^{-1}(\tan(\alpha))=\alpha -\pi$ if $\alpha\in (\pi/2,\pi)$, we find that  $\varphi_1\circ \bPsifh(\beta,\alpha)=(\beta,\alpha)$ for $\alpha\in (0,\pi)$, and thus $\bPsifh|_{(0,\pi)}$ is a smooth diffeomorphism onto its image.
    If $\alpha\in (-\pi,0)$ so that $\bPsifh(\beta,\alpha)\in U_2$, a similar computation shows that, again, $\varphi_1\circ \bPsifh(\beta,\alpha)=(\beta,\alpha)$ (now one needs to use that $\tan^{-1}(\tan(\alpha))=\alpha +\pi$ if $\alpha\in(-\pi,-\pi/2)$).
    We conclude that $\bPsifh$ is a smooth bijection with smooth inverse, hence a diffeomorphism.
\end{proof}

The scattering map extends in a natural way to a map
\begin{equation}
    \overline{\mathcal{S}}^H:\overline{\Gamma}\to \overline{\Gamma} ,\quad \overline{\mathcal{S}}^H([(\beta,{a},\lambda) ])=\begin{cases}
        [{\mathcal{S}}^H(\beta,a,\lambda)],\quad & a\in \Rm,\\
        [(\beta,\pm\infty,\lambda)], \quad
        &
        {a}=\pm \infty,
    \end{cases}
\end{equation}
and the intertwining property \eqref{eq:intertwining_HE} extends to all of $\partial S\Dm_E$:
\begin{equation}\label{eq:ext_scat}
    \bPsihf\circ \overline{\mathcal{S}}^H=  \mathcal{S}^E\circ\bPsihf.
\end{equation}
Since $\S^E:\partial S\Dm_E\to \partial S\Dm_E$ is a smooth diffeomorphism fixing $\partial_0 S\Dm$ and $\overline{\mathcal{S}}^H=\bPsifh\circ  \mathcal{S}^E\circ \bPsihf$, Lemma \ref{lm:diffeo} immediately implies the following.

\begin{proposition}
	The extended scattering map $\overline{\mathcal{S}}^H: \overline{\Gamma}\to \overline{\Gamma}$ is a smooth diffeomorphism whose (pointwise) fixed point set is given by $\overline{\Gamma}_0 \coloneqq \overline{\Gamma}_+/_\sim\cap \overline{\Gamma}_-/_\sim$.
\end{proposition}

Now, the following alternative characterization of the space $C_{\alpha,+}^\infty (\Ghbar)$ is immediate by \eqref{eq:CalphaG}.
\begin{proposition}\label{prop:ambient}
	One has
\begin{equation}
	C_{\alpha,+}^\infty (\Ghbar) = \{u\in C^\infty(\overline{\Gamma}_+) \colon A_{+}^H u \in C^\infty(\overline{\Gamma}),\ (\overline{\S}_A^H)^* u =  u\}.
\end{equation}	
where $\Ghbar$ is identified with $\overline{\Gamma}_+$ as before, $\overline{\S}_A^H=\overline{\S}^H\circ \overline{A}_H$, and $\overline{A}_H:\overline{\Gamma}\to \overline{\Gamma}$ is the extension of $A_H$ to $\overline{\Gamma}$ by continuity.
\end{proposition}

\subsection{Proof of theorem \ref{thm:gammazeroHalt} }\label{ssec:gammazeroH}

Before proving Theorem \ref{thm:gammazeroHalt}, we note the following lemma.

\begin{lemma}\label{lm:HTDNH}
    For all fixed $s\ge 0$, one has
    \begin{equation}
	H_{T,D ,\pm}^s (\Ghbar)=\muh \Psihf^{*}H_{T,D ,\pm}^s(\partial_+S\Dm_E),\quad
	H_{T,N ,\pm}^s (\Ghbar)=\muh \Psihf^{*}H_{T,N ,\pm}^s(\partial_+S\Dm_E).
    \end{equation}
\end{lemma}

Now, we have the following proof.
\begin{proof}[Proof of Theorem \ref{thm:gammazeroHalt}]
    Fix $s\ge 0$. For $u$ as in the statement, we have  $\Psihf^{-*}\muh^{-1}u\in \mu C_{\alpha,+}^\infty(\partial_+S\Dm_E)$ (resp. $H_{T,D,+}^{s+1/2}(\partial_+S\Dm_E)$, by Lemma \ref{lm:HTDNH}). By Theorem \ref{thm:rangegammazero}, Theorem \ref{thm:rangeP}, and \eqref{eq:PU_Smooth}, the following are equivalent:
    \begin{enumerate}[(i)]
	\item\label{eq_cond_0alt} There exists $\tilde{f}\in C^\infty(\Dm_E)$ (resp.  $\tilde{f}\in \widetilde{H}^{s,0}(\Dm_E))$ such that $\Psihf^{-*}\muh^{-1}u=I^E_0\tilde{f}$;
	\item \label{eq_cond_1alt} $C_-\Psihf^{-*}\muh^{-1} u=0$;
	\item\label{eq_cond_2alt} There exists $\tilde{w}\in C_{\alpha,-}^\infty(\partial _+S\Dm_E)$ (resp. $H_{T,N,-}^{s+1/2}(\partial_+ S\Dm_E)$) such that $P_-\tilde{w}=\Psihf^{-*}\muh^{-1}u$.
    \end{enumerate}
    Now, \ref{eq_cond_0alt} is equivalent to the existence of $f=\Phi^{*}\tilde{f}\in \Cev (\Dm_H)$ (resp. $\wtH^{s,0}_w (\Dm_H)$) such that
    \begin{equation}
	\Psihf^{-*}\muh^{-1}u=I_0^E\tilde{f}=I_0^E\Phi^{{-*}}{f}\iff u=\muh\Psihf^{*}I_0^E \Phi^{-*}f\overset{\text{Thm. \ref{thm:ProjEq}}}{=}I_0^Hx^2 f,
    \end{equation}
    that is, \ref{eq_cond_0alt} is equivalent to \ref{gammazeroHit0alt}.
    On the other hand, \ref{eq_cond_2alt} holds exactly when ${C}_-^Hu=0$ by Proposition \ref{prop:intertwining_bd_ops}, so it is equivalent to \ref{gammazeroHit1alt}.
    Finally, the equation in \ref{eq_cond_2alt} can be rewritten as
    \begin{equation}
	u=\muh \Psihf^{*} P_-\tilde{w}\overset{\text{Prop.\ \ref{prop:intertwining_bd_ops}}}{\iff} u ={P}_-^H\mu_h \Psihf^{*}\tilde{w}\iff u=P_-^Hw,
    \end{equation}
    where $w \coloneqq \muh\Psihf^{*}\tilde{w}\in\muh C_{\alpha,-}^\infty(\Ghbar)$ (resp. $w\in H_{T,N,-}^{s+1/2} (\Ghbar)$), so \ref{eq_cond_2alt} is equivalent to \ref{gammazeroHit2alt}.  We have thus shown {\ref{gammazeroHit0alt}$\Leftrightarrow$\ref{gammazeroHit1alt}$\Leftrightarrow$\ref{gammazeroHit2alt}}. Finally, \ref{gammazeroHit0alt}$\Leftrightarrow$\ref{gammazeroHit3alt} is a special case of Theorem \ref{thm:moments} for $\gamma=0$, using that  $(\sigma_{n,k}^0)^2=\frac{4\pi}{n+1}$.
\end{proof}

%%%%%%%%%%%%%%%%%%%%%%%%%%%%%%%%%%%%%%%%%%%%%%%%%%%%%%%%%%%%%%%%%%%%%%%%%%%%%%%%%%%%%%%%%%%%%%%%%%%%%%%%%%%%

\appendix

\section{Hyperbolic moment conditions} \label{sec:HMC}

In this section, we discuss the hyperbolic moment conditions as defined in \cite{Berenstein1993a}, and show in particular that they are equivalent to \eqref{eq:remoments} in the present setting, when choosing $p_m(x)= (-x)^m$.
There, the authors state the hyperbolic moment conditions for rapidly decaying functions on the space $\G_{n,k}$ of totally geodesic $k$-dimensional planes in the $n$-dimensional Poincar\'e ball. They use the parameterization $\G_{n,k}=\{(\xi,v,r)\in  \Xi_{n,k}\times\Sm^{n-1}\times [0,\infty):v\in \xi\}$, where
\begin{equation}
    \Xi_{n,k}=\big\{\Sm^{n-1}\cap \Pi:\;\; \Pi \text{ a } n-k\text{-dimensional plane through the origin in }\Rm^n \big\}.\label{eq:Xink}
\end{equation}
Each triple in $\G_{n,k}$ corresponds to the $k$-geodesic plane orthogonal to the plane determined by $\xi$, through the point $rv$, where $r$ denotes geodesic distance from the origin (see \cite[p.4]{Berenstein1993a}).
Then, the hyperbolic moment conditions for a rapidly decaying function $\phi$ on $\G_{n,k}$ are stated as follows: For every $\xi\in \Xi_{n,k} $ and $v\in \xi$,
\begin{equation}\label{eq:MoCoBeTa}
    \int_{v\in \xi}\langle v ,v'\rangle_E^{m}
    \left[\int_0^{\infty}\frac{\tanh^{m+n-k-1}(r)}{\cosh(r)}\phi(\xi,v,r)\d r\right]\d S( v)=P_m(v').
\end{equation}
Here, $\d S(v)$ is the measure on $\xi\cong \Sm^{n-k-1}$ induced by the Lebesgue measure on $\Pi$ in \eqref{eq:Xink}, and $P_m$ is a homogeneous polynomial of degree $m$.

In our case, $n=2$, $k=1$, and each $\xi=\Sm^1\cap L$, where $L$ is a line through the origin, consists of two points.
So, let $\xi=\{\pm v_0\}\in \Xi_{2,1}$, where $v_0\in \Sm^1$, and rewrite \eqref{eq:MoCoBeTa} as
\begin{equation}\label{eq:MoCoBeTa2}
    \sum_{v=\pm v_0}\langle v ,v'\rangle_E^{m}
    \int_0^{\infty}\frac{\tanh^{m}(r)}{\cosh(r)}\phi(\xi,v,r)\d r=P_m(v'),\qquad v'=\pm v_0.
\end{equation}
If \eqref{eq:MoCoBeTa2} holds for $v'\in \xi$, it also holds for $-v'$ (by homogeneity of $P_m$). So, let $v'=v_0$. Then, \eqref{eq:MoCoBeTa2} is equivalent the statement that for all $v_0\in \Sm^1$,
\begin{equation}\label{eq:MoCoBeTa3}
\begin{split}
    P_m(v_0)=&\sum_{v=\pm v_0}\langle v ,v_0\rangle_E^{m}	\int_0^{\infty}\frac{\tanh^{m}(r)}{\cosh(r)}\phi(\xi,v,r)\d r\\
    =&\int_0^{\infty}\frac{\tanh^{m}(r)}{\cosh(r)}\phi(\xi,v_0,r)\d r+(-1)^m\int_0^{\infty}\frac{\tanh^{m}(r)}{\cosh(r)}\phi(\xi,-v_0,r)\d r.
\end{split}
\end{equation}
If we view $\phi$ as a function on the double cover $\G_{2,1}'=\{(\xi,v,r)\in  \Xi_{2,1}\times\Sm^{1}\times (-\infty,\infty):v\in \xi\}$ of the space of (unoriented) geodesics, we see that  $\phi(\xi,v_0,r)=\phi(\xi,-v_0,-r)$.
Therefore,
\begin{equation}\label{eq:MoCoBeTa4} 	
\begin{split}	
    P_m(v_0)\overset{\phantom{\rho=-r}}{=}&\int_0^{\infty}\frac{\tanh^{m}(r)}{\cosh(r)}\phi(\xi,v_0,r)\d r+(-1)^m\int_0^{\infty}\frac{\tanh^{m}(r)}{\cosh(r)}\phi(\xi,v_0,-r)\d r\\
    \overset{\rho=-r}{=}&\int_0^{\infty}\frac{\tanh^{m}(r)}{\cosh(r)}\phi(\xi,v_0,r)\d r+(-1)^m\int_{-\infty}^{0}\frac{(-1)^m\tanh^{m}(\rho)}{\cosh(\rho)}\phi(\xi,v_0,\rho)\d\rho\\
    \overset{\phantom{\rho=-r}}{=}&\int_{-\infty}^{\infty}\frac{\tanh^{m}(r)}{\cosh(r)}\phi(\xi,v_0,r)\d r.
\end{split}
\end{equation}
With $\x(t)$ defined in \eqref{eq:hyp_fan_beam}, since $\x(t)$ is a unit-speed geodesic, the geodesic distance from $0$ to $\x(t)$ is $t$. By rotation-invariance, this means that the Euclidean distance $s$ from the origin appearing in our vertex parameterization is related to the geodesic distance $r$ by $s = \tanh(r/2)$, and this implies $\tanh(r)=\frac{2s}{1+s^2}$ and $\cosh(r)=\frac{1+s^2}{1-s^2}$. Therefore, changing variable $r\to s$, we arrive at
\begin{equation}
\begin{split}
	P_m(v_0)=&\int_{-\infty}^{\infty}\cosh(r)\tanh^{m}(r)\phi(\xi,v_0,r)\mathrm{sech}^2(r)\d r\\
	=&\int_{-1}^{1}\frac{1+s^2}{1-s^2}\left(\frac{2s}{1+s^2}\right)^{m}\phi\big(\xi,v_0, \tanh^{-1} \Big(\frac{2s}{1+s^2}\Big)\big)\left(2\frac{1-s^2}{(1+s^2)^2}\right)\d s\\
	=&\int_{-1}^{1} \left(\frac{2s}{1+s^2}\right)^{m}\phi^{\mathrm{v}}(\arg(v_0),s)\ \frac{2\d s}{1+s^2},
\end{split}
\end{equation}
and this recovers \eqref{eq:remoments} (using the correspondence of $\gamma^{\mathrm{v}}_{\omega,s}$, viewed as an unoriented geodesic, with $ \big(\{\pm e^{i\omega}\},e^{i\omega},\tanh^{-1}(\frac{2s}{1+s^2})\big)\in \G_{2,1}'$).

\section*{Acknowledgments}
The authors thank the anonymous referees, whose comments and suggestions improved the presentation of the article.
%We would like to thank you for \textbf{following
%the instructions above} very closely. It will save us lot of time and expedite the
%process of your article's publication.

%%%%%%%%%%%%%%%%%%%%%%%%%%%%%%%%%%%%%%%%%%%%%%%%%%%%%%
%          7. REFERENCES SECTION
%%%%%%%%%%%%%%%%%%%%%%%%%%%%%%%%%%%%%%%%%%%%%%%%%%%%%%

%       READ THIS SECTION CAREFULLY

% Each of the references below MUST be cited in your article above. Do not include references that are not cited in your article.

% Follow the examples below carefully. We strongly suggest that you copy and paste your reference information directly into our examples.

% List all references in alphabetical order according to the first author's last name.

% Verify each URL works correctly and can be accessed properly. Your URL links should be to reputable websites. The command line for a website link begins with: \url{ }

% Do not add MR or DOI numbers to your references. AIMS production staff will add this information.

% Using BibTex is not recommended but can be handled.

\bibliographystyle{plain}
% \bibliography{bibliography}

\end{document}